\documentclass[reqno,11pt]{amsart}
\usepackage{geometry}
\usepackage[numbers,sort&compress]{natbib}
\usepackage{graphicx}
\usepackage{amsfonts,amsmath,amssymb,amsthm,mathtools}
\usepackage{bbm}
\usepackage[foot]{amsaddr}
\usepackage[colorlinks=true,citecolor=blue]{hyperref}
\usepackage{pgfplotstable}
\usepackage{tikz}
\usetikzlibrary{arrows,positioning,chains,fit,shapes,calc,decorations}
\usepackage[font=small, labelfont=rm]{caption} % Small font for main caption
\usepackage[font=small, labelformat=simple]{subcaption} % Small font for subcaptions
\newcommand{\Fone}{F_1}
\newcommand{\Gauss}{\,{}_2F_1}
\newcommand{\ii}{\mathrm i}
\newcommand{\D}{\mathcal D}
\newcommand{\st}{s_{\text{\sc v}}}
\newcommand{\ts}{s_{\text{\sc h}}}
\newcommand{\lip}{\ell\hspace{0.25ex}}
\newcommand{\lipalpha}{\ell_\alpha}
\newcommand{\salpha}{s_\alpha}
\newcommand{\zalpha}{u_\alpha}
\newcommand{\palpha}{r}
\newcommand{\pp}{\lambda}
\newcommand{\pq}{q}
\newcommand{\sr}{K}
\newcommand{\rew}{\omega}

\newcommand{\norm}[1]{\left\Vert #1 \right \Vert}
\newcommand{\cm}{c}
\newcommand{\bm}{R}
\newcommand{\EE}{\mathbb{E}}
\newcommand{\PP}{\mathbb{P}}

\newcommand{\ZZ}{\mathbb{Z}}
\newcommand{\RR}{\mathbb{R}}
\newcommand{\N}{\mathbb{N}}

\newcommand{\Rtight}{S}

\DeclareMathOperator{\Bin}{Bin}
\DeclareMathOperator{\Ber}{Bernoulli}

\newtheorem{theorem}{Theorem}
\newtheorem{proposition}[theorem]{Proposition}

\newtheorem{definition}{Definition}
\newtheorem{lemma}{Lemma}
\newtheorem{corollary}[theorem]{Corollary}

\title[Krasnosel'skii--Mann iterations beyond asymptotics]{Krasnosel'skii--Mann iterations beyond asymptotics:\\ a combinatorial analysis}

\author{Mario Bravo$^{1}$}
\address{$^1$Departamento de Administraci\'on, Facultad de Administraci\'on y Econom\'ia,
Universidad de Santiago de Chile, Chile}
\email{mario.bravo.g@usach.cl}
\author{Roberto Cominetti$^{2}$}
\address{$^2$Institute for Mathematical and Computational Engineering and
Department of Industrial and Systems Engineering,
Pontificia Universidad Cat\'olica de Chile}
\email{roberto.cominetti@uc.cl}

\begin{document}

\begin{abstract}
We revisit the classical Krasnosel'skii--Mann fixed point iteration for contractions and nonexpansive maps in general normed spaces. This iteration is ubiquitous across a wide range of areas, including convex optimization, monotone inclusions, Markov decision processes, under-relaxed methods for nonlinear PDEs, and more.
Drawing on a remarkable connection with a Markov chain on $\mathbb{Z}^2$, and using counting arguments from  enumerative combinatorics of lattice paths, we derive explicit estimates for the distance between iterates, as well as non-asymptotic error bounds for the fixed point residuals.
As the contraction parameter approaches one, these bounds smoothly recover the known estimates for nonexpansive maps. Building upon these estimates, we further derive error bounds for inexact Krasnosel'skii--Mann iterations.
\end{abstract}

\subjclass[2020]{Primary 47J25, 47J26, 65J15; Secondary 33C05, 33C65, 33C90, 65K15}
\keywords{fixed point iterations, contractive maps, error bounds, lattice paths, under-relaxed Picard iterations, hypergeometric functions}

\maketitle

\vspace{-3ex}
\section{Introduction}
Let $C$ be a convex set in a general normed space $(X,\norm{\cdot})$, and  $T:C\to C$ an $\lip$-Lipschitz map with $\lip\in (0,1]$.
As is well known, for strict contractions ($\lip<1$) the  Banach--Picard iteration $x_{m+1}=Tx_{m}$ started from an arbitrary $x_0 \in C$
satisfies the geometric error bound
\begin{equation}\label{eq:bound_p}
\|x_m-x_{m+1}\|=\|x_m-Tx_m\|\leq \|x_0-Tx_0\|\,\lip^m.
\end{equation}
Thus $(x_m)_{m\in\N}$ is a Cauchy sequence and, if $X$ is complete and $C$ is closed, its limit $x^*$ exists and is the unique fixed point of $T$. A triangle inequality then yields $\|x_{m} - x^*\|\leq \|x_m-Tx_m\|/{(1-\lip)}$.

When $T$ is just nonexpansive ($\lip=1$) both the existence and the uniqueness of fixed points can fail, and the bound \eqref{eq:bound_p} becomes uninformative. For this case, \citet{krasnosel1955two} and
\citet{mann1953mean} introduced the following under-relaxed iteration with $\alpha \in (0,1]$ and $x_0 \in C$,
\begin{equation*} \tag{$\textsc{km}$} \label{eq:km_a}
 (\forall m\in\N)\qquad   x_{m+1}=(1-\alpha)\,x_{m}+\alpha\,Tx_{m}.
\end{equation*}
This is just Banach--Picard's iteration for the averaged map  $T_\alpha x=(1-\alpha)\,x + \alpha\, T x$ and then, since  $\|x - Tx\|= \|x - T_\alpha x\|/\alpha$ and $T_\alpha$ 
 is $\lipalpha$-Lipschitz with $\lipalpha=1-\alpha + \alpha \,\lip$, it follows that
\begin{equation} \label{eq:bbbound_p2}
 \|x_m - Tx_m\|\leq \|x_0-Tx_0\|\,\lipalpha^{\hspace{0.3ex}m}.
\end{equation} 

 This bound is weaker than \eqref{eq:bound_p} because $\lipalpha\ge \lip$, and remains uninformative when $\lip=1$. However, a finer analysis by \citet{baillion1996rate} showed that for nonexpansive maps on domains with  finite diameter $\mathop{\rm diam}(C)\le\kappa$,  for $\alpha\in(0,1)$ and $m\geq 1$ the residuals of \eqref{eq:km_a}  satisfy the vanishing bound
   \begin{equation} \label{eq:BB_constant}
    \|x_m-Tx_m\|\leq \kappa\hspace{0.2ex}\Gauss\!\left(\mbox{$\frac12$},-m;2;4\alpha(1\!-\!\alpha)\right)\le \frac{\kappa}{\sqrt{\pi \alpha(1\!-\!\alpha)m}},
    \end{equation}
    where $\Gauss$ denotes Gauss's hypergeometric function. 
    The hypergeometric expression extends continuously to $\alpha=1$, where it equals 1; the square-root estimate is not intended at that endpoint.

For $\lip<1$ the geometric bound \eqref{eq:bound_p} is smaller than \eqref{eq:BB_constant} if $m$ is large enough. However, in the nearly nonexpansive regime $\lip\approx 1$ the distinction between contractive and nonexpansive dynamics becomes increasingly blurred,
and the rate of the Banach--Picard iteration deteriorates to the point that a large number of iterations might be required before   \eqref{eq:bound_p} overtakes \eqref{eq:BB_constant}. 
A natural question ---and the main motivation for this paper--- is whether a sharper analysis of \eqref{eq:km_a} in the contractive setting can produce bounds that decay geometrically when $\lip<1$, while smoothly recovering \eqref{eq:BB_constant} as $\lip \to 1$.
A related question is how to choose the relaxation parameter 
$\alpha$ in order to obtain the smallest possible error bound, and how all of this changes when the iteration \eqref{eq:km_a} is subject to errors.

Beyond their intrinsic interest, these questions are motivated by applications in various domains. 
In discounted Markov decision processes, the value operator is a $\|\cdot\|_\infty$-contraction with 
$\lip$ the discount factor, and long-horizon problems correspond to the regime 
$\lip\approx 1$ in which convergence becomes slow and explicit error bounds are especially relevant (see  \citet{puterman1994markov} and \citet{sutton1998reinforcement}). Likewise, many algorithms in convex optimization, variational inequalities, and monotone inclusions, can be seen as fixed point iterations of proximal or resolvent maps. When strong convexity or strong monotonicity is present but weak, these maps are contractive yet nearly nonexpansive, and the iteration \eqref{eq:km_a} arises naturally (see \citet{bauschke2011convex}). Another area where \eqref{eq:km_a} shows up is the evolution of lazy random walks on graphs and metric spaces (see 
\citet{levin_peres_wilmer_2017}, \citet{ollivier2010survey} and  \citet{woess2000random}).
In a different domain, numerical solutions of nonlinear PDEs by 
finite-element, finite-volume, and finite-difference methods often use under-relaxed iterations of the form \eqref{eq:km_a}  to improve robustness and damp oscillations
(see \citet[\S{}4.5]{patankar1980numerical}; \citet[\S{}8.9 and \S{}14]{moukalled2016finite};
\citet[\S{}5.4, \S{}7 and  \S{}8]{ferziger_computational_2020}; \citet[\S{}5]{langtangen2017finite}). 
In particular, under-relaxation is implemented in the CFD software {\em OpenFOAM} and in the multiphysics simulation tool {\em MOOSE}, for both Picard and Newton iterations. As noted in these references, there is no general rule for selecting 
a suitable relaxation  $\alpha$, 
which is problem dependent and has to be adjusted 
empirically.

Taken together, these examples show that Krasnosel'skii--Mann  schemes are not just a  variation of the classical contraction framework, but a useful and still not fully understood tool that bridges the strong convergence properties of contractive maps with the subtler phenomena characteristic of nonexpansive iterations.
We hope  our results contribute to a better understanding of how
and when under-relaxation can be beneficial, and provide some guidance on how to select the relaxation parameter $\alpha$ in \eqref{eq:km_a}.

\subsection{Summary of the paper}

We place our study in an abstract framework with $(X,\|\cdot\|)$ a normed space and $T:C \to C$ an $\lip$-Lipschitz map with $\lip\in (0,1]$, on a convex subset  $C\subseteq X$ of finite diameter $\mathop{\rm diam}(C)\le\kappa$.  Throughout we consider a fixed $\alpha\in (0,1]$ (including $\alpha=1$) and we denote
\begin{equation*}
\left\{\begin{aligned}
\lipalpha&\triangleq 1-\alpha+\alpha\,\lip,\\
\sigma&\triangleq\sqrt{\lip}.
\end{aligned}\right.
\end{equation*}

Our main goal is to bound the fixed point residuals for the \eqref{eq:km_a} iteration. 
Since $\|x_m-Tx_m\|=\|x_m-x_{m+1}\|/\alpha$, our strategy is to estimate the distance between iterates.
The core argument, inspired by the nonexpansive case \cite{baillion1996rate,csv14}, relies on a family of recursive bounds $\|x_m - x_n\|\leq \kappa\,\cm_{m,n}$ defined  for any pair of iterates. Interpreting the recurrence as absorption probabilities for a specific Markov chain in $\ZZ^2$, and  using counting arguments from enumerative combinatorics of lattice paths, we derive analytic formulas for these $\cm_{m,n}$'s.
Our first main result, proved in Section \ref{sec:2}, is as follows.

\begin{theorem}\label{thm:main} Let $(x_m)_{m\geq 0}$ be the sequence generated by \eqref{eq:km_a} starting from $x_0\in C$. Then, for all $0\leq m\le n$ we have $\norm{x_m-x_n}\le\kappa\,\cm_{m,n}$ and $\|x_m-Tx_m\|\leq\kappa\,\cm_{m,m+1}/\alpha$, where
\begin{align}
\label{eq:cmn}
\cm_{m,n}&=\alpha\sum_{j=0}^{n-1} \sum_{k=0}^m
L^{k,j}_{m,n}\,(\alpha^2\lip)^k(1\!-\!\alpha)^{n+m-1-j-2k},\\
\label{eq:latticepaths}
L^{k,j}_{m,n}&= \binom{m}{k}\!\binom{n\!-\!1\!-\!j}{k} - \binom{n}{k}\!\binom{m\!-\!1\!-\!j}{k}.
\end{align}
\end{theorem}

\noindent{\sc Remark.} In these formulas, and throughout the paper, we follow  Krattenthaler's convention \cite{krattenthaler2015} with $\binom{a}{k}=0$
 if $k$ is not in the range $0\le k \leq a$, and in particular if $a<0\leq k$. Hence, the inner sum in \eqref{eq:cmn} can be restricted to $0\le k\le \min\{m,n-1-j\}$. This avoids negative powers of $(1-\alpha)$ which are undefined for $\alpha=1$. Notice that the second
negative term in \eqref{eq:latticepaths} already vanishes for $k\geq m-j$. We also adopt the standard conventions $\binom{0}{0}=1$ and $0^0=1$.

\vspace{1ex}

Theorem \ref{thm:main} still holds if $C$ is unbounded, provided one has an {\em a priori} bound $\|x_0-Tx_m\|\leq\kappa$ for all $m\in\N$. In particular, this holds for $\kappa=(1+\lip)\|x_0-x^*\|$ if $T$ has a fixed point $x^*$,
and also with $\kappa=\|x_0-Tx_0\|+\mathop{\rm diam}(T(C))$ when $T$ has bounded range.

Hereafter we denote $\bm_m\triangleq \cm_{m,m+1}/\alpha$, so that the error bound reads $\|x_m-Tx_m\|\leq\kappa\,\bm_m$.
Using the  identities in Appendix \ref{App:Hypergeom}, which relate Gauss's $\Gauss$ and Appell's $\Fone$ hypergeometric functions, Proposition \ref{Thm:Alt_cmn} presents a first characterization for $\alpha\in (0,1)$
\begin{equation}\label{Eq:Rmbis}
\mbox{$\bm_m=(1-\alpha)^{2m}\Fone(-m;-(m+1),2;2;(\frac{\alpha}{1-\alpha})^2\lip,-\frac{\alpha}{1-\alpha})$},
\end{equation}
whereas for $\alpha=1$ formula \eqref{eq:cmn}
yields $\bm_m=\lip^m$.

The different plots in Figure \ref{figura4} show that for $\lip\approx 1$ there is a 
range of indices $m$ in which the bound $\kappa\,\bm_m$ is significantly smaller than the  coarse bound $\kappa\, \lip^m$ for  Banach--Picard. This comparison is not entirely fair since the prefactor $\|x_0-Tx_0\|$ in \eqref{eq:bound_p} can be much smaller than $\kappa$. However, if $\|x_0-Tx_0\|$ is a constant fraction of $\kappa$, \eqref{eq:km_a} may still achieve smaller error bounds
when $\lip\approx 1$. 
Note also that one can exploit the  faster initial decrease of the \eqref{eq:km_a} bound and later switch to Banach--Picard from a better starting point.

\begin{figure}[ht!]
  \centering
  \begin{minipage}{.3\linewidth}
    \centering
    \includegraphics[width=\linewidth]{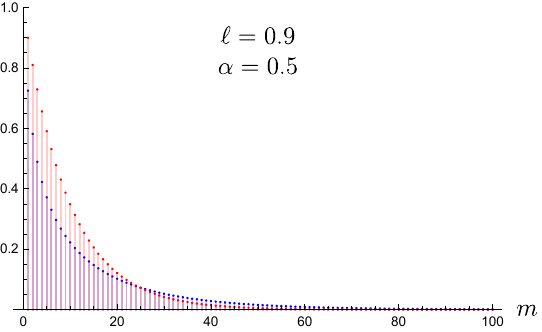}
  \end{minipage}
  \begin{minipage}{.3\linewidth}
    \centering
    \includegraphics[width=\linewidth]{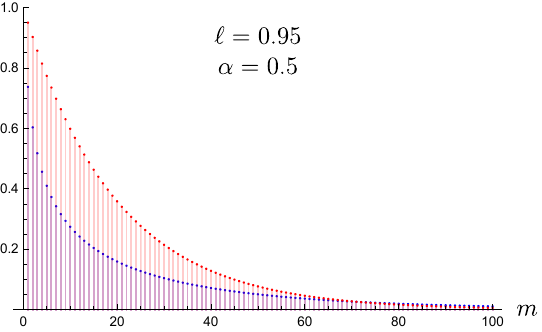}
  \end{minipage}
  \begin{minipage}{.3\linewidth}
    \centering
    \includegraphics[width=\linewidth]{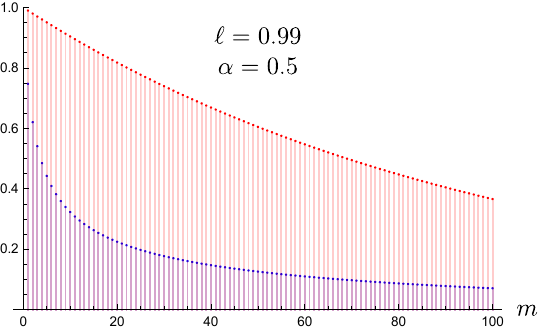}
  \end{minipage}

   \begin{minipage}{.3\linewidth}
    \centering
    \includegraphics[width=\linewidth]{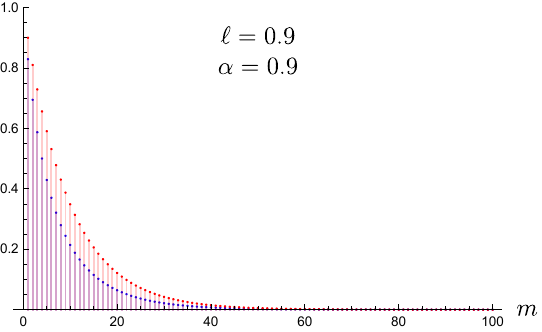}
  \end{minipage}
  \begin{minipage}{.3\linewidth}
    \centering
    \includegraphics[width=\linewidth]{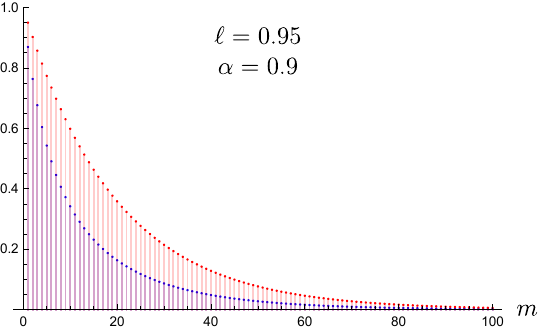}
  \end{minipage}
  \begin{minipage}{.3\linewidth}
    \centering
    \includegraphics[width=\linewidth]{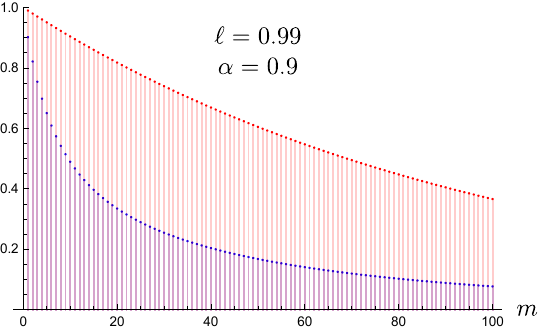}
  \end{minipage}
 
  \caption{
Krasnosel'skii--Mann bound $\kappa\,\bm_m$ (blue) and  Banach--Picard coarse  bound $\kappa\,\lip^m$  
  (red) as functions of the iteration 
 $m$, for $\kappa=1$ and different relaxation parameters $\alpha$ and contraction factors $\lip$. Already for 
 $\lip=0.9$ the \eqref{eq:km_a} bound improves upon  Banach–Picard's coarse bound, particularly with $\alpha=0.9$. The improvement becomes increasingly pronounced as the contraction factor  $\lip$ approaches 1. Note however that comparing the upper bounds does not exclude that for specific maps either method can outperform the other.}
  \label{figura4}
\end{figure}

We stress that, except for the initial recursion that uses the triangle inequality to define the $\cm_{m,n}$'s, the identities \eqref{eq:cmn} and \eqref{Eq:Rmbis}, as well as the subsequent representations, are obtained as equalities. Nevertheless, $\bm_m$ is  just an upper bound for the residuals. In order to evaluate how sharp this bound is, in \S{}\ref{Sec:optimal_transport} we compare $\bm_m$ with a known minimax-tight bound $\Rtight_m$ derived from a family of  optimal transport problems (see \cite{bc18,bravo2022universal,bravo_cominetti_lee2026halpern}). 
For $\alpha\in [\frac12, 1]$, Proposition \ref{cmndmn} establishes the explicit absolute-gap estimate  
$\bm_m-\Rtight_m\leq (1-\alpha)^2(m+1)\lipalpha^m/\alpha$. The numerical results in Appendix 
\ref{AppE_Numerical} suggest substantially
smaller absolute gaps and also small relative gaps $\bm_m/\Rtight_m$, with the advantage that $\bm_m$ admits explicit analytic formulas whereas $\Rtight_m$ is defined by an involved recursion.

While the previous formulas for $\cm_{m,n}$ and $\bm_m$ have a nice probabilistic interpretation, rooted in the combinatorial structure behind \eqref{eq:km_a}, their polynomial 
nature makes them difficult to analyze.
Section \ref{sec:3} derives alternative probabilistic and integral representations for the $c_{m,n}$'s, more amenable to deriving upper bounds and asymptotic rates,  as well as for determining an $\alpha$ that minimizes $\bm_m$. Namely, considering two independent counting processes $X_n=B_1+\ldots+B_n$ with
$B_i\sim\mathop{\rm Bernoulli}(\alpha)$, and $Y_n=D_1+\cdots+D_n$ with $D_i\sim\mathop{\rm Bernoulli}(r)$ where $r=\alpha\,\lip/\lipalpha$, Proposition \ref{Prop4} shows that
\begin{equation*}
\cm_{m,n}
=
\lipalpha^{\hspace{0.3ex}m}\,\PP(X_n\!>\! Y_m)
-
\lipalpha^{\hspace{0.2ex}n}\,\PP(X_m\!>\!Y_n).
\end{equation*}
Then, using Cauchy's residue theorem, in Proposition \ref{prop:int} we derive the integral formula
 \begin{equation*}
\cm_{m,n}
=
\lipalpha^{\hspace{0.3ex}m}-\lipalpha^{\hspace{0.2ex}n}
+
\frac{2\sigma}{\pi}
\int_0^\pi
\frac{
r_{\!\alpha}(\theta)^{m+n}
\sin\!\bigl((n\!-\!m)\,\varphi_{\!\alpha}(\theta)\bigr)\sin\theta
}{
1-2\sigma\cos\theta+\sigma^2
}\,d\theta,
\end{equation*}  
 where $r_{\!\alpha}(\theta)\triangleq |z_{\alpha}(\theta)|$ and $\varphi_{\!\alpha}(\theta)\triangleq\arg(z_{\alpha}(\theta))$ with $z_{\alpha}(\theta)=1-\alpha+\alpha\,\sigma\, e^{\ii\theta}$. We recall that $\sigma=\sqrt{\lip}$.
 This leads to our second main result.
\begin{theorem}\label{Prop6}
Let $r_{\!\alpha}(\theta)=|z_{\alpha}(\theta)|$ as above. Then, for all $m\geq 0$ we have
\begin{equation}\label{eq:cmmp1_Chebyshev00}
\bm_{m}
=
(1-\lip)\,\lipalpha^{\hspace{0.3ex}m}
+
\frac{2\lip}{\pi}
\int_0^\pi
\frac{r_{\!\alpha}(\theta)^{2m}\sin^2\theta}{1-2\sigma\cos\theta+\sigma^2}\;d\theta.
\end{equation}
Moreover, denoting 
$\salpha\triangleq (1-\alpha+\alpha\,\sigma)^2$ and $\eta\triangleq 4\alpha(1\!-\!\alpha)\sigma/\salpha$, we also have
\begin{equation}\label{eq:cmmp1_Chebyshev2}
\bm_{m}=(1-\lip)\,\lipalpha^{\hspace{0.3ex}m}+
\left(\frac{8\lip}{\pi}
\int_0^1
\frac{\sqrt{x(1-x)}\;(1-\eta\,x)^{m}}{(1-\sigma)^2+4\sigma\,x}\,dx\right) \salpha^{\hspace{0.2ex}m},
\end{equation}
 which can be expressed using Euler’s integrals 
\eqref{Eq:Euler_2F1}-\eqref{Eq:Euler_F1}, with $\xi\triangleq 4\sigma/(1\!-\!\sigma)^2$ when $\lip<1$,  as
\begin{equation}\label{eq:cmmp1_Chebyshev5}
\bm_{m}
=\left\{\begin{array}{ll}
(1-\lip)\,\lipalpha^{\hspace{0.3ex}m}
+
\frac{\lip}{(1-\sigma)^2}\,\Fone\big(\mbox{$\frac{3}{2}$};-m,1;3;\eta,-\xi\big)\,\salpha^{\hspace{0.2ex}m}&\text{ if }\lip<1,\\[2ex]
{}\Gauss\!\left(\mbox{$\frac12$},-m;2;4\alpha(1\!-\!\alpha)\right)&\text{ if }\lip=1.
\end{array}\right.
\end{equation}
\end{theorem}

The formula \eqref{eq:cmmp1_Chebyshev5} is continuous at $\lip=1$,  thereby answering our original question: the bound $\bm_m$ decays geometrically when $\lip<1$, and smoothly recovers the nonexpansive bound \eqref{eq:BB_constant} as $\lip\to 1$. Note also that for $\alpha=1$ these expressions collapse to $\bm_m=\lip^m$, matching the Banach--Picard bound \eqref{eq:bound_p} with $\|x_0-Tx_0\|$ bounded above by $\kappa$.

The integral formulas \eqref{eq:cmmp1_Chebyshev00}-\eqref{eq:cmmp1_Chebyshev2} are more tractable than \eqref{Eq:Rmbis}: they readily provide simpler upper bounds and asymptotic rates (Corollary \ref{prop7}), and allow us to optimize $\bm_m=\bm_m(\alpha)$ as a function of $\alpha$ for a fixed number  
of iterations $m$.
In \S{}\ref{sec:5} we show that for each $m\geq 1$ there is a unique minimizer 
$\alpha_\lip(m)$ and we analyze its asymptotics for $m\to\infty$, as well as its dependence on $\lip$.  For $\lip\leq\frac12$ we have $\alpha_\lip(m)\equiv 1$ for all $m$ which yields a standard Banach--Picard iteration, whereas
for  $\lip\in(\frac12,1)$ we have  $\alpha_\lip(m)\in(\frac12,1)$ with $\alpha_\lip(m)\approx 1-c^*_\lip/m\to 1$ as $m\to\infty$, for a specific constant $c^*_\lip>0$. In contrast, for fixed  $m$ we have $\alpha_\lip(m)\to\alpha_1(m)=\frac12$ as $\lip\uparrow 1$.

Building upon the previous results, the final Section \ref{sec:6} considers the case of inexact Krasnosel'skii--Mann iterations in which the operator values $Tx_m$ can only be computed up to an additive error $e_m$. For this case, Theorem \ref{Thm12} establishes a finite-time error bound in the form of a discrete convolution between the sequences $\bm_m$ and $\|e_m\|$.

\subsection{Related work}

The iteration \eqref{eq:km_a} is the special case of the more general scheme
\begin{equation}\label{eq:km_b}
x_{m+1}=(1-\alpha_{m})\,x_m+\alpha_{m} Tx_m,
\end{equation}
with nonconstant $\alpha_m\in (0,1]$.
This iteration has a rich history, mainly in the nonexpansive case $\lip=1$, starting with the landmark works of \citet{krasnosel1955two} and  \citet{mann1953mean} in the early 1950s. For the specific case $\alpha_m\equiv\alpha =\frac12$, on uniformly convex  spaces and assuming  $T(C)$ precompact,
Krasnosel'skii established the strong convergence of the iterates towards a fixed point. This was extended by \citet{Schaefer1957} to all $\alpha \in (0,1)$ and by \citet{edelstein1966remark} to strictly convex spaces (see also \citet{outlaw1969mean}). The extension to general Banach spaces was achieved by \citet{ishikawa1976fixed} with $\alpha_m$ bounded away from 1 and $\sum_{m\geq 0}\alpha_m=\infty$, and also by  \citet{edelstein1978nonexpansive} for constant $\alpha_m\equiv\alpha$. 
Without compactness, but assuming
the existence of fixed points and $\sum_{m\geq 0}\alpha_m(1\!-\!\alpha_m)=\infty$,  \citet{reich1979weak} proved weak convergence of the iterates in uniformly convex spaces with a smooth norm.

The crucial step in all these convergence results
was to show that $\|x_m-Tx_m\|\to 0$, a property that was dubbed {\em asymptotic regularity} and established under different assumptions by \citet{browder1966solution}, \citet{dotson1970mann}, \citet{groetsch1972note}, \citet{baillon1978asymptotic}, \citet{goebelkirk1983},  \citet{reich1990nonexpansive} and \citet{borwein1992krasnoselski}.  This sequence of results culminated in an
explicit metric estimate by \citet{baillon1992optimal},
who conjectured the existence of a universal constant
$c$ such that
$$\|x_m-Tx_m\|\leq \frac{c\,\mathop{\rm diam}(C)}{\sqrt{\sum_{i=1}^m\alpha_i(1-\alpha_i)}}.$$
Soon after, \citet{baillion1996rate} proved that for  $\alpha_m\equiv\alpha$ this holds with $c=1/\sqrt{\pi}$. The general case with nonconstant parameters $\alpha_m$ was confirmed in \citet{csv14} with a different proof but the same constant $1/\sqrt{\pi}$, while \citet{bc18} later showed that this constant is tight. 

 \citet{halpern1967fixed} introduced a similar iteration $x_{m+1}=(1-\beta_m)x_0+\beta_m Tx_m$ with $\beta_m\uparrow 1$.
It has been studied extensively, with a renewed  recent interest focused on error bounds, lower bounds, and optimal $\beta_m$'s. For nonexpansive maps, a suitable choice of $\beta_n$ achieves $\|x_m-Tx_m\|\leq 4 \kappa/(m+1)$, which decays faster than the order $\Theta(1/\sqrt{m})$ characteristic of Krasnosel'skii--Mann 
(see \citet{sabach_shtern2017first},  \citet{lieder2021convergence},  \citet{bravo2022universal}). As shown in \citet{contreras2023optimal},  the  optimal choice is to set recursively $\beta_{m+1}=(1+\beta_m^2)/2$ with $\beta_0=0$, which yields a slightly 
smaller and minimax-tight bound.
 For strict contractions, the optimal Halpern 
 iteration was obtained by  \citet{park2022exact} in Hilbert spaces, and recently by  \citet{bravo_cominetti_lee2026halpern} 
 in normed spaces.
 The latter builds on ideas from \cite{bc18,bravo2022universal} exploiting a family of tight recursive bounds based on optimal transport, which will be further discussed later in connection with \eqref{eq:km_a}.

\section{Recursive estimates and error bounds}
\label{sec:2}

This section establishes the estimates  $\|x_m-x_n\|\leq\kappa\, \cm_{m,n}$ in Theorem \ref{thm:main} for the distance between iterates. By symmetry, it suffices to consider $m\leq n$ with $c_{n,n}=0$. This is the basic building block for all subsequent results,
providing in particular the fixed point residual bound $\|x_m-Tx_m\|\leq \kappa\,\bm_m$  with $\bm_m\triangleq \cm_{m,m+1}/\alpha$. We will also compare these estimates with the known minimax optimal bounds 
 obtained through a sequence of recursive optimal transport problems,
 revealing that the $\bm_m$'s not only admit a closed form analytic expression but they track closely the minimax-tight bounds. In the subsequent sections \S\ref{sec:3} and \S{}\ref{sec:4} we will present alternative formulas for the quantities $\cm_{m,n}$ and the bounds $\bm_m$, providing more manageable error bounds and asymptotic rates.

We split the proof of Theorem \ref{thm:main} into several auxiliary lemmas.
We begin by stating a family of recursive bounds $\cm_{m,n}$. These $\cm_{m,n}$'s are then interpreted as absorption probabilities for a special Markov chain in $\ZZ^2$, from which we derive the explicit analytic form \eqref{eq:cmn}-\eqref{eq:latticepaths}. Although our main interest is in the contractive case $\lip<1$, the analysis also applies to nonexpansive maps with $\lip=1$, providing alternative expressions for the results in \citet{baillion1996rate}.

\subsection{Recursive estimates}\label{sec:recursive}
A useful way to analyze the \eqref{eq:km_a} iterates is to express them as convex combinations of the previous images: considering the weights $\pi_0^m=(1-\alpha)^m$ and
$\pi_i^m= \alpha\,(1-\alpha)^{m-i}$ for  $1\leq i\leq m$, 
and adopting the convention $Tx_{-1}=x_0$, we have
\begin{equation*}
 x_m= \sum_{i=0}^m \pi_i^m\, T x_{i-1}.
\end{equation*}
If we already have  bounds $\norm{x_i-x_j}\leq \kappa\,\cm_{i,j}$ for $0\leq i<j<n$, the $\lip$-Lipschitz property of $T$ together with  the identity $\pi_i^m-\pi_i^n=\pi_i^m\sum_{j=m+1}^n\pi_j^n$ for $0\le i\le m$ and the coarse estimate $\norm{x_0-Tx_{j-1}}\leq\kappa$, allow us to derive a  valid bound for $\norm{x_m-x_n}$ as
 \begin{eqnarray*}
    \norm{x_m-x_n}&=&\norm{\mbox{$\sum_{i=0}^m\pi_i^m\,Tx_{i-1}-\sum_{j=0}^n\pi_j^n\,Tx_{j-1}$}}\\
     &=&\norm{\mbox{$\sum_{i=0}^m\sum_{j=m+1}^n\pi_i^m \pi_j^n(Tx_{i-1}-Tx_{j-1})$}}\\
&\leq&\kappa\,\mbox{$(\pi_0^m -\pi_0^n) + \kappa\sum_{i=1}^m\sum_{j=m+1}^n\lip\,\pi_i^m \pi_j^n\cm_{i-1,j-1}$}.
\end{eqnarray*}
Thus, we obtain $\norm{x_m-x_n}\leq \kappa\,\cm_{m,n}$ by  defining recursively
\begin{equation}\label{eq:recursive}
(\forall \, 0\leq m\leq n)\qquad\cm_{m,n}\triangleq \pi_0^m-\pi_0^n + \sum_{i=1}^m\sum_{j=m+1}^n\!\!\lip\,\pi_i^m \pi_j^n\cm_{i-1,j-1}.
\end{equation} 
To prove Theorem~\ref{thm:main} we will interpret this recursion as absorption probabilities 
for a particular Markov chain, and then show that
the explicit solution  is given by
\eqref{eq:cmn}-\eqref{eq:latticepaths}. 

\vspace{1.5ex}
\noindent{\sc Remarks.}\\[0.5ex]
$a)$  We have $\cm_{n,n}=0$ and  $\cm_{0,n}=1-\pi_0^n$. Starting from this, one can show that \eqref{eq:recursive} is equivalent to the simpler recursion for $1\leq m< n$
\begin{equation}\label{eq:recurrencia_corta}
\cm_{m,n}=(1\!-\!\alpha)\,\big(\cm_{m,n-1}+\cm_{m-1,n}\big)+\big(\lip\,\alpha^2-(1\!-\!\alpha)^2\big)\,\cm_{m-1,n-1}.
\end{equation}
$b)$ The estimates $\|x_m\!-x_n\|\le\kappa\,\cm_{m,n}$ remain valid for maps defined on unbounded domains, provided that one has an {\em a priori} bound $\|x_0-Tx_m\|\leq\kappa$. In particular, this holds when $T$ has a bounded range
with $\kappa=\|x_0-Tx_0\|+\mathop{\rm diam}(T(C))$, or when $T$ has a fixed point $x^*$ with $\kappa=(1+\lip)\|x_0-x^*\|$.\\[0.5ex]
$c)$ 
Although we present the results for normed spaces, they extend to geodesic spaces such as CAT(0) and Busemann  spaces, and more generally to convex metric spaces in the sense of Takahashi (see Appendix \ref{App:ConvexMetricSpace}). If $T:C\to C$ is $\lip$-Lipschitz with $C\subseteq X$ bounded and geodesically convex,  the iterates 
$x_{m+1}=(1-\alpha)\,x_m\oplus\alpha\, Tx_m$ defined by geodesic averages instead of convex combinations, satisfy the same recursive bounds $d(x_m,x_n)\leq\kappa\, \cm_{m,n}$, hence the subsequent analysis remains valid (except the results on additive errors which require a linear structure). For completeness we include the proof in Appendix \ref{App:ConvexMetricSpace}; a simple adaptation of 
\citet[Proposition 4.4]{foglia_colao_2025_cat0}.

\subsection{An absorbing Markov chain}\label{Sec:2.2}
Consider a Markov chain with state space $\mathcal S=\D\cup\{\st,\ts\}$ with $\D\triangleq \{(m,n)\in \mathbb Z^2: 0\leq m\le n \}$ and two auxiliary absorbing states $\st$ and $\ts$. The transition probabilities are $P(\st;\st)=P(\ts;\ts)=1$,
while from $(m,n)\in\D$ we take (notice the shift of indices from $i-1$ to $i$ and $j-1$ to $j$)
\begin{equation}\label{Eq:probP}
\begin{cases}
  P((m,n); (i,j))= \lip \,\pi_{i+1}^m \pi_{j+1}^n \qquad\text{for } i=0,\ldots, m-1\text{ and } j=m,\ldots,n-1\\
   P((m,n);\ts)=\pi_{0}^m -\pi_{0}^n\\
P((m,n);\st)=\mbox{$1-(\pi_{0}^m -\pi_{0}^n)-\lip\sum_{i=0}^{m-1}\sum_{j=m}^{n-1}\pi_{i+1}^m\pi_{j+1}^n$}.
\end{cases}
\end{equation}
Notice that, since $\sum_{j=m}^{n-1}\pi_{i+1}^m\pi_{j+1}^n=\pi_{i+1}^m\!-\pi_{i+1}^n$ and $\lip\leq 1$, we have $P((m,n);\st)\geq \sum_{i=0}^{m}\pi_{i}^n\geq 0$.
Moreover, for $n=m$ we have $P((m,m);\st)=1$ and the chain is immediately absorbed at $\st$, while for $m=0<n$ we have $P((0,n);\st)=\pi_0^n$ and  $P((0,n);\ts)=1-\pi_0^n$.

This chain can be interpreted as a game between a vertical player $V$ and a horizontal player $H$, whose joint positions describe a random path in $\ZZ^2$ (see Figure~\ref{fig:1a}). 
The game proceeds in {\em phases} in which players move alternately performing a sequence of Bernoulli steps with continuation probabilities $\beta\triangleq 1-\alpha$. Absorption at $\st$ is encoded by states  $(m,n)$  below the diagonal with $n<m$, and absorption at $\ts$  by states of the form $(-1,n)$. The game proceeds as follows.

\vspace{2ex}

\begin{itemize}\itemsep=2ex
\item At the beginning of a phase the players are located at some state $(m,n)\in \D$. First,  player  $V$ moves down by 
a sequence of independent Bernoulli unit steps with continuation probabilities $\beta$, except from $(m,0)$ to $(m,-1)$ which stops with probability 1. With probability
    $\pi_{j+1}^n$, player $V$ stops at height $j$ just below the first failure.
    If $V$ lands below the diagonal $j<m$, she wins and the game ends (absorption at $\st$).
   
\item If $V$'s new position is on  the diagonal or above, that is  $j\geq m$, the phase continues with player $H$ moving left by a sequence of Bernoulli 
steps with continuation probabilities $\beta$, except from $(0,j)$ to $(-1,j)$ which again stops with probability 1. With probability $\pi_{i+1}^m$, player $H$ stops at the position $i$ after the first failure. In the event $i=-1$ player $H$ wins  and the game ends (absorption at $\ts$, at level $j$).
    
\item If the game has not ended, the players are now located at $(i, j)$ with $0\leq i < j$. At this point, player $V$ gets a second chance: a Bernoulli variable $B$ with parameter $\lip \in (0,1]$ is drawn. If $B = 0$ the game ends and $V$ is declared the winner, 
otherwise the current phase is called {\em unsettled} and a new phase starts from  $(i, j)$.
\end{itemize}

\begin{figure}[ht]
    \centering
    \begin{subfigure}[b]{0.49\textwidth}
        \centering
    \begin{tikzpicture}[scale=0.5,every node/.style={scale=0.6}]
  \draw[step=1cm,gray,very thin] (-2,-2) grid (10,10); 
 \draw[thick,->] (0,0) -- (10,0);
 \draw[thick,->] (0,0) -- (0,10);
 \foreach \i in {0,1,...,8}{
	\pgfmathsetmacro{\Start}{\i+1}
 	\foreach \j in {\Start,...,9} 
 	\draw[black,fill=black] (\i,\j) circle (0.05cm);
}

\draw[dashed,-] (-2,-2) -- (10,10);
\node at (7,9.65) {$(m,n)$};
\node at (-1,2.35) { $(-1,j)$};
\draw[ultra thick,-stealth,color=blue] 
  (7,9)--(7, 7) -- (5, 7) 
  -- (5, 6) 
  -- (1, 6) 
  -- (1,3)
  --(-1,3);
  \draw[ultra thick,-stealth,color=red] 
  (1, 6) 
  -- (1,3)
  --(-1,3);
\draw[black,fill=blue] (5,7) circle (0.15cm);
\draw[black,fill=blue] (1,6) circle (0.15cm);
\draw[black,fill=red] (-1,3) circle (0.15cm);
\end{tikzpicture}
\caption{ $H$-winning path from $(m,n)=(7,9)$, with $k=2$ unsettled phases (blue) and final phase absorbed at level $j=3$ (red). Fat dots mark the end of a phase. The probability of this path is $\alpha\,(\lip\, \alpha^2)^2\, (1-\alpha)^{8}$.}
 \label{fig:1a}
    \end{subfigure}
    \hfill
     \begin{subfigure}[b]{0.49\textwidth}
        \centering
      \begin{tikzpicture}[scale=0.5,every node/.style={scale=0.6}]
  \draw[step=1cm,gray,very thin] (-2,-2) grid (10,10); 
 \draw[thick,->] (0,0) -- (10,0);
 \draw[thick,->] (0,0) -- (0,10);
 \foreach \i in {0,1,...,8}{
	\pgfmathsetmacro{\Start}{\i+1}
 	\foreach \j in {\Start,...,9} 
 	\draw[black,fill=black] (\i,\j) circle (0.05cm);
}
 \foreach \i in {0,1,...,8}{
	\pgfmathsetmacro{\Start}{\i+1}
 	\foreach \j in {\Start,...,9} 
 	\draw[black,fill=black] (\j,\i) circle (0.05cm);
}

\draw[dashed,-] (-2,-2) -- (10,10);
\node at (7,9.65) {$(m,n)$};
\node at (-1,2.35) { $(-1,j)$};
\draw[ultra thick,dotted]  (7,9)--(7, 8);
\draw[ultra thick] 
(7, 8) --(7,7)-- (5, 7) 
  -- (5, 6) 
  -- (1, 6) 
  -- (1,3)
  --(0,3);
  \draw[ultra thick, dotted,-{stealth}] (0,3)--(-1,3);
  \node at (8.35,6.45) {$(n\!-\!1,m)$};
\node at (3,-0.35) { $(j,0)$};
\draw[ultra thick,color=black!70,-{stealth}] 
--(3,0)-- (3,1)--(6, 1)-- (6, 5)--(7,5)--(7,7)--(10,7);
\draw[black,fill=black] (5,7) circle (0.15cm);
\draw[black,fill=black] (1,6) circle (0.15cm);
\draw[black,fill=black] (-1,3) circle (0.15cm);
\draw[black,fill=gray!50] (7,5) circle (0.15cm);
\draw[black,fill=gray!50] (6,1) circle (0.15cm);

\end{tikzpicture}
        \caption{Reflecting the path across the diagonal $x=y$, 
        and ignoring the first and last steps, yields a lattice path below the diagonal (grey) from $(j,0)=(3,0)$ to $(n-1,m)=(10,7)$  with $k=2$ east-north turns.}
         \label{fig:1b}
    \end{subfigure}
    \caption{An $H$-winning path and its reflected lattice path}
    \label{fig:1}
\end{figure}
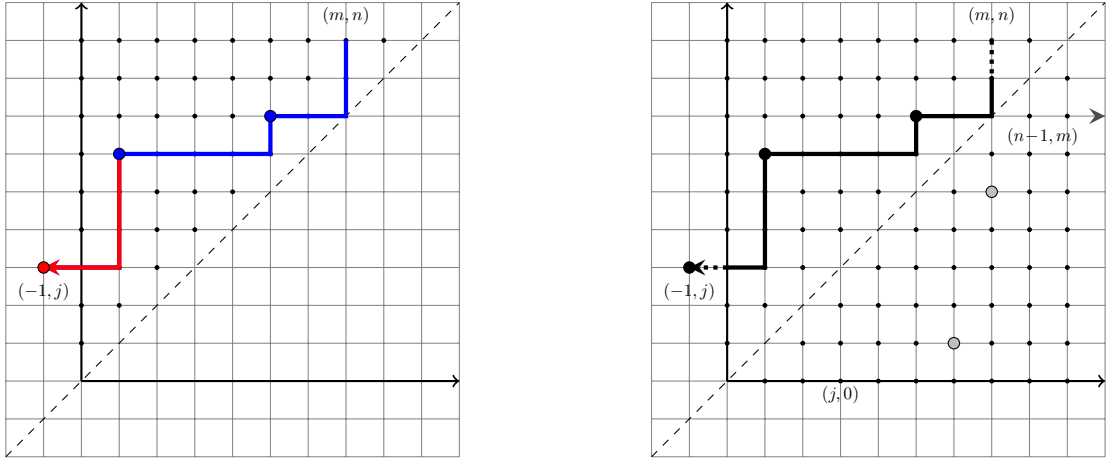

\begin{lemma} \label{lem:interp2}
For all $0\leq m< n$ we have
$\cm_{m,n}= \mathbb P \left ( \text{\normalfont $H$ wins $|\,s_0=(m,n)$} \right )$.
\end{lemma}
\begin{proof}
    Let $p_{m,n}$ be the probability that the chain is absorbed at the $H$-winning state $\ts$, when starting from $s_0=(m,n)\in\D$. By
conditioning on the first jump it follows that $p_{m,n}$ satisfies exactly the same recursion \eqref{eq:recursive}, with $p_{0,n}=1-\pi_0^n=\cm_{0,n}$, and therefore $\cm_{m,n}=p_{m,n}$ for all $(m,n)\in\D$.
\end{proof}

\subsection{Stratification of \texorpdfstring{$H$}{Lg}-winning paths}
\label{sec:stratification}
By partitioning the set of all $H$-winning paths
into the subsets $\mathcal{P}^{k,j}_{\!m,n}$ of 
paths from $(m,n)\in\D$ to a specific $H$-winning position $(-1,j)$ and which have exactly $k$ unsettled phases, we have
\begin{equation}\label{eq:sumprob}
\cm_{m,n}=\sum_{j=0}^{n-1} \sum_{k=0}^m \PP(\mathcal{P}^{k,j}_{\!m,n}).
\end{equation}
Note that each unsettled phase moves at least one unit down and at least one unit left, so that for any fixed $j$ the number of unsettled phases is at most $k\leq \min\{m,n-1-j\}$, and for larger $k$ the set $\mathcal{P}^{k,j}_{\!m,n}$ is empty.
The crucial observation is that all paths  in  $\mathcal{P}^{k,j}_{\!m,n}$ have exactly the same probability.
\begin{lemma}\label{prop:proba}
     Let $0\leq m<n$ and $0\leq k\leq \min\{m,n-1-j\}$ with $0\le j\le n-1$. Then, the probability of any path in $\mathcal{P}^{k,j}_{\!m,n}$ is exactly $\alpha \,(\alpha^2\lip)^k\,(1\!-\!\alpha)^{n+m-1-j-2k}$. 
\end{lemma}
\begin{proof} 
In all $H$-winning paths, player $H$ moves from $m$ to $-1$ for a total of $m+1$ steps.
Each unsettled phase includes one failure contributing a factor $\alpha^k$,
 the last step in the final phase fails with probability 1, and the remaining $m-k$ steps are successes. Altogether the horizontal steps contribute a factor $\alpha^k(1\!-\!\alpha)^{m-k}$ to the probability. Similarly, on each path in $\mathcal{P}^{k,j}_{\!m,n}$  player $V$ makes $n-j$ vertical moves with exactly $k+1$ failures (one in each phase, including the final phase),
 contributing a factor $\alpha^{k+1}(1\!-\!\alpha)^{n-j-(k+1)}$.
Finally, each unsettled phase adds a factor $\lip$ for a total of $\lip^k$, so the result follows by multiplying these three factors.
\end{proof}

\subsection{Lattice path counting}\label{sec:latticepaths}
In view of Lemma \ref{prop:proba}, in order to compute $\PP(\mathcal{P}^{k,j}_{\!m,n})$ it remains to count the number of paths in $\mathcal{P}^{k,j}_{\!m,n}$. To this end, we observe that a reflection across the diagonal of $\ZZ^2$ establishes a bijection with a family of simple lattice paths below the diagonal (see Figure~\ref{fig:1b}). This allows us to leverage known results on  enumeration of lattice paths, for which we refer to~\citet{krattenthaler2015}. Below we recall the precise definition and the specific result needed for our purposes. 
\begin{definition}
    A {\em simple lattice path} is a sequence $(P_1,P_2,\ldots, P_N)$ of points in $\ZZ^2$ such that for all $1\leq i<N$ the increments $\Delta_i=P_{i+1}-P_{i}$ are either $(1,0)$ or $(0,1)$, so  that the path takes horizontal or vertical unit steps in the positive direction.
An {\em east-north turn} occurs when the path moves horizontally with $\Delta_i=(1,0)$ followed immediately after by a vertical step $\Delta_{i+1}=(0,1)$.
\end{definition}
\begin{theorem}[{{\cite[Theorem 10.14.1]{krattenthaler2015}}}] \label{thm:kratt}
    Let $a\geq b$ and $c \geq d$. The number of all simple lattice paths from $(a,b)$ to $(c,d)$ staying weakly below the diagonal $x = y$ with exactly $k$ east-north turns is
$$EN_{k}[(a,b);(c,d)]=\binom{c\!-\!a}{k}   \binom{d\!-\!b}{k} - \binom{c\!-b\!+\!1}{k}  \binom{d\!-\!a\!-\!1}{k},$$
with the convention $\binom{n}{k}=0$ when $k$ is not in the range $0\le k \leq n$, and in particular if $n<0\leq k$.
\end{theorem}

\subsection{Proof of Theorem \ref{thm:main}}
\begin{proof} The case $n=m$ is immediate, so let us assume $n>m$. All paths in $\mathcal{P}^{k,j}_{\!m,n}$ have an initial vertical step
from $(m,n)$ to $(m,n-1)$ and a final
horizontal step from $(0,j)$ to $(-1,j)$. Ignoring these segments, and reflecting across the diagonal $x=y$, establishes a bijection with the lattice paths from $(j,0)$ to $(n-1,m)$ that remain below the diagonal with  exactly $k$ east-north turns. Theorem~\ref{thm:kratt} with $(a,b)=(j,0)$ and $(c,d)=(n-1,m)$ gives
$$|{\mathcal P}^{k,j}_{\!m,n}|=\binom{m}{k}\!\binom{n\!-\!1\!-\!j}{k} - \binom{n}{k}\!\binom{m\!-\!1\!-\!j}{k}=L^{k,j}_{m,n}.$$ 
Notice that $L^{k,j}_{m,n}$ is a nonnegative integer. This, combined with Lemma \ref{prop:proba}, implies 
\begin{equation}\label{Eq:Pmnkj}
\PP({\mathcal P}^{k,j}_{\!m,n})=L^{k,j}_{m,n}\,\alpha\,(\lip\,\alpha^2)^k(1\!-\!\alpha)^{n+m-1-j-2k}.
\end{equation}
Plugging this in \eqref{eq:sumprob} yields the announced formula \eqref{eq:cmn} for $\cm_{m,n}$. For $n=m+1$ this readily implies
$\|x_m-Tx_m\|=\|x_m-x_{m+1}\|/\alpha\leq\kappa\,\cm_{m,m+1}/\alpha$.
\end{proof}

\noindent{\sc Remark.}
An alternative proof of Theorem \ref{thm:main} consists in checking that the right hand side of \eqref{eq:cmn} satisfies the recursion \eqref{eq:recursive}, or its 
equivalent \eqref{eq:recurrencia_corta}.
However, such algebraic proof would obscure the origin of the formula
for $\cm_{m,n}$. To the best of our knowledge, this formula is new even in the nonexpansive case $\lip=1$ for which it provides a simpler, though equivalent, expression for the one given in \cite[Theorem 4.2]{baillion1996rate}.

\subsection{Representation by hypergeometric functions}
The algebraic identities in Appendix \ref{App:Hypergeom}, which relate Gauss's $\Gauss$ and Appell's $\Fone$ hypergeometric functions,
allow us to express $\cm_{m,n}$ and $\bm_m$ in the following alternative forms.
\begin{proposition}\label{Thm:Alt_cmn} Let $\alpha\in(0,1)$ and set $\zalpha=(\frac{\alpha}{1-\alpha})^2\lip$. Then, for all $0\leq m<n$ we have
\begin{align}
    \cm_{m,n}&=\mbox{$\alpha(1-\alpha)^{m+n-1}\left(n\,\Fone(1\!-\!n;-m,1;2;\zalpha,-\frac{\alpha}{1-\alpha})-m\,\Fone(1\!-\!m;-n,1;2;\zalpha,-\frac{\alpha}{1-\alpha})\right)$},    \label{Eq:cmnNuevo}\\
   \bm_m&= \mbox{$(1-\alpha)^{2m}\Fone(-m;-(m+1),2;2;\zalpha,-\frac{\alpha}{1-\alpha})$}.    \label{Eq:RmNuevo}
\end{align}
\end{proposition}
\begin{proof} For $m=0$ the formula \eqref{Eq:cmnNuevo} follows directly from $c_{0,n}=1-\pi_0^n$. Let us then consider the case $m\geq 1$. The equality \eqref{Eq:Pmnkj} can be written as 
  $\PP(\mathcal{P}^{k,j}_{\!m,n})=\alpha\,(1\!-\!\alpha)^{n+m-1-j}\,L^{k,j}_{m,n}\,\zalpha^{\hspace{0.2ex}k}$. 
Now, the definition of $\Gauss$ gives
$\Gauss(-a,-b;1;z)=\sum_{k\geq 0}\binom{a}{k}\binom{b}{k}\,z^k$, 
so that substituting $z=\zalpha$ and observing that the negative term in $L_{m,n}^{k,j}$ vanishes for $j\geq m$, we obtain 
$$\sum_{k\geq 0}\PP(\mathcal{P}^{k,j}_{\!m,n})=
\alpha (1\!-\!\alpha)^{m+n-1-j}\left(\Gauss(j+1-n,-m;1;\zalpha)-\mathbbm{1}_{\{j<m\}}\Gauss(j+1-m,-n;1;\zalpha)\right).
$$
Then,
\eqref{Eq:cmnNuevo} follows by plugging this into \eqref{eq:sumprob} and applying twice the identity \eqref{eq:main-identity} in Appendix \ref{App:Hypergeom} with $t=1/(1-\alpha)$, first for $a=n-1$ and then with $a=m-1$ (for $m=0$ the second term vanishes). 

Evaluating \eqref{Eq:cmnNuevo} for $n=m+1$, the identity \eqref{Eq:RmNuevo} is a direct consequence of   the following contiguous relation for $\Fone$
(see for instance \citet[equation (6a)]{Schlosser2013Appell})
\begin{align*}
   a\Fone(a+1;b_1,b_2;c;u,v)-(a-b_1-b_2)\Fone(a;b_1,b_2;c;u,v)&=\\
   b_1\Fone(a;b_1+1,b_2;c;u,v)+b_2\Fone(a;b_1,b_2+1;c;u,v),
\end{align*}
with $a=-m$, $b_1=-(m+1)$, $b_2=1$, and $c=2$.
\end{proof}

\subsection{Comparison with tight error bounds based on optimal transport}
\label{Sec:optimal_transport}
 Krasnosel'skii--Mann iteration is a special case of general Mann schemes (see \cite{mann1953mean}) defined recursively
  from 
$x_0 \in C$ by
\begin{equation*} \tag{$\textsc{m}$} \label{eq:mann}
\mbox{$(\forall m\in\N)\quad x_m=\sum_{i=0}^m\pi^m_i\,Tx_{i-1}$}
\end{equation*}
for a triangular array of averaging scalars 
$\pi^m_i\ge 0$ with $\sum^m_{i=0}\pi^m_i=1$, $\pi^m_m>0$, and $\pi^m_i= 0$ for $i>m$. As before, we use the convention $Tx_{-1}\!= x_0$.

By extending the nonexpansive analysis in \cite{bc18,bravo2022universal}, tight error bounds for Lipschitz maps were obtained in \cite[Appendix A]{bravo_cominetti_lee2026halpern}. For contractions these bounds are as follows.
 Let  $\mathcal{F}_{m,n}$ be the set of transport plans from $\pi^m$ to $\pi^n$, that is, vectors $z=(z_{i,j})_{i=0,\dots, m; j=0,\dots, n}$ with $z_{i,j} \ge 0$ and 
$$\left\{\begin{array}{cl}
    \mbox{$\sum^n_{j=0} z_{i,j}$} = \pi^m_i &\mbox{ for } 0\leq i\le m,\\
        \mbox{$\sum^m_{i=0} z_{i,j}$} = \pi^n_j & \mbox{ for } 0\le j\le n.
\end{array}\right.$$
Starting with $d_{-1,-1}=0$ and $d_{-1,n}=d_{n,-1}=1/\lip$ for all $n\in\N$, define $d_{m,n}$  by the recursive family of optimal transport problems
\begin{equation*}\label{Eq:optrans}
(\forall m,n\in\N)\quad d_{m,n} \triangleq \min_{z \in \mathcal{F}_{m,n}}\sum^m_{i=0}\sum^n_{j=0}  z_{i,j}\,\lip\,d_{i-1,j-1}.
\end{equation*}
Combining these $d_{m,n}$'s, define $\Rtight_m\triangleq\sum^m_{i=0} \pi^m_i\,\lip\,d_{i-1,m}$.
\begin{theorem}[\cite{bravo_cominetti_lee2026halpern}]
    Every sequence $(x_m)_{m\in\N}$ generated by \emph{\eqref{eq:mann}} satisfies $\|x_m-x_n\|\leq \kappa\,d_{m,n}$ and $\|x_m-Tx_m\|\leq\kappa\,\Rtight_m$, for all $m,n\in\N$. These bounds are minimax tight: 
    for every $\lip\in(0,1]$, $\kappa\geq 0$, and triangular sequence $(\pi^n)_{n\in\N}$ of averaging factors, there is   a normed space $(X,\|\cdot\|)$, an $\lip$-Lipschitz map  $T:C\to C$ on a closed convex set $C\subseteq X$ with diameter $\kappa$, and a corresponding Mann sequence $(x_n)_{n\in\N}$ 
     that satisfies all these bounds with equality.
\end{theorem}

The scalars $\Rtight_m$ and $d_{m,n}$ depend solely on $\lip$ and the $\pi^n$'s, and the bounds are valid for every $\lip$-Lipschitz map $T:C\to C$.
As shown in \cite{bravo_cominetti_lee2026halpern}, the $d_{m,n}$'s define a metric in $\N$
with $d_{m,n}\in[0,1]$, and each $d_{m,n}$ admits a {\em simple} optimal transport such that $z_{i,i}=\min\{\pi_i^m,\pi_i^n\}$. 
We supplement this by the following monotonicity and strengthening of the triangle inequality. We omit the proofs, which follow step by step the arguments  for the nonexpansive case  in \cite[Theorems 5.1 and 5.2]{bravo2022universal}. 

\begin{theorem}
Suppose that the averaging factors satisfy the monotonicity condition
    $$(\forall n\geq 1)\quad \mbox{$\pi_n^n>0$ and $0\leq \pi_i^{n}\leq\pi_i^{n-1}$ for $0\le i<n$.}\leqno\mbox{\sc(h)}$$
Then, in the region $m\leq n$ the  $d_{m,n}$'s are nonincreasing in $m$ and nondecreasing in $n$. Moreover, they satisfy the {\em convex quadrangle inequality} 
\begin{equation*}\label{eq:4point}
d_{i,l}+d_{j,k}\leq d_{i,k}+d_{j,l}\qquad   \mbox{ for all $i\leq j\leq k\leq l$.}\leqno \mbox{\sc(q)}
\end{equation*}
\end{theorem}

\vspace{-3ex}
\begin{figure}[ht] 
\centering
\begin{tikzpicture}[thick,scale=0.6,-,shorten >= 1pt,shorten <= 1pt,every node/.style={scale=0.5}]
\begin{scope}[start chain=going right,node distance=15mm]
 \node[on chain,draw,circle] (m)  {$i$};
 \node[on chain,draw, circle] (k)  {$j$};
 \node[xshift=2cm,on chain,draw, circle] (j)  {$k$};
\node[on chain,draw,circle] (n) {$l$};
\end{scope}
\path[every node/.style={font=\sffamily\small}] (m) edge [bend left = 35]   node [fill=white] {$d_{i,l}$}  (n) ;
\path[every node/.style={font=\sffamily\small}] (k) edge [bend left]   node  [fill=white]  {$d_{j,k}$} (j) ;
\path[every node/.style={font=\sffamily\small}] (k) edge [bend right = 35]  node  [near end,sloped, fill=white] {$d_{j,l}$} (n) ;
\path[every node/.style={font=\sffamily\small}] (m) edge [bend right = 35]   node  [near start,sloped,fill=white] {$d_{i,k}$} (j) ;

\end{tikzpicture}
\vspace{-1ex}
\end{figure}

A remarkable consequence of {\sc(q)} is that each $d_{m,n}$ admits an optimal transport that is not only {\em simple} but also {\em nested}, in the sense that there is no flow crossing: there are no $i<j<k<l$ with $z_{i,k}>0$ and $z_{j,l}>0$. Such a nested transport plan is uniquely determined, and can be computed 
by a greedy procedure (see the Inside-Out Algorithm in \citet{baillon1992optimal, bravo2022universal}).

\vspace{2ex}
Now let us consider the optimal transport bounds $d_{m,n}$ for iteration \eqref{eq:km_a}, where $\pi_0^m=(1-\alpha)^m$ and $\pi^m_i=\alpha\,(1-\alpha)^{m-i}$ for $1\le i\le m$.
Since {\sc (h)} is clearly satisfied,
each $d_{m,n}$ has a nested optimal transport. In particular, for  $\alpha\geq\frac12$ and $n>m$ the optimum has the explicit form (see Figure \ref{PQtransports})
\begin{equation}\label{trasporte_optimo}
z_{i,j}=\left\{
\begin{array}{cl}
\pi_i^n&\text{for }i=j\in\{0,\ldots,m\}\\
\pi_i^m-\pi_i^n&\text{for }i\in\{0,\ldots,m-1\}\text{ and }j=n\\
\pi^n_j&\text{for } i=m \text{ and } j\in\{m+1,\ldots,n-1\}\\
\pi^m_m-\sum_{j=m}^{n-1}\pi_j^n&\text{for } i=m \text{ and } j=n\\
0&\text{otherwise}
\end{array}
\right.
\end{equation}
while for $n=m$ we simply have $z_{i,i}=\pi_i^m$ for all $i\in\{0,\ldots,m\}$ with all the other flows equal to 0. Observe that $z_{m,n}=\pi^m_m-\sum_{j=m}^{n-1}\pi_j^n=(2\alpha-1)+\pi_0^{n+1-m}\geq 0$ because $\alpha\geq\frac12$.

\begin{figure}[t]
\centering
\begin{tikzpicture}[thick,-,shorten >= 1pt,shorten <= 1pt,scale=0.7,every node/.style={scale=0.7}]
 \node[draw,circle] (i0) at (0,0.0) {$0$}; \node[draw,circle] (j0) at (5.5,0.0) {$0$};
 \node[] (idots) at (0,-0.7)   {$\vdots$};  \node[] (jdots) at (5.5,-0.7)   {$\vdots$};
 \node[draw,circle] (ii) at (0,-1.7) {$i$}; \node[draw,circle] (ji) at (5.5,-1.7) {$i$};
 \node[] (idots2) at (0,-2.4)   {$\vdots$};  \node[] (jdots2) at (5.5,-2.4)   {$\vdots$};
 \node[draw,circle] (im) at (0,-3.4) {$m$}; \node[draw,circle] (jm) at (5.5,-3.4) {$m$};

 \node[] (jdots3)  at (5.5,-4.1) {$\vdots$};
 \node[draw,circle] (jj) at (5.5,-5.1) {$j$};
 \node[] (jdots4)  at (5.5,-5.8) {$\vdots$};
 \node[draw,circle] (jn) at (5.5,-6.8) {$n$};
 \draw[->] (i0)--(j0);
 \draw[->] (ii)--(ji);
 \draw[->] (im)--(jm);
 \draw[->] (i0)--(jn);
 \draw[->] (ii)--(jn);
 \draw[->] (im)--(jn);
  \draw[->] (im)--(jj);

  \node[draw,circle] (ii0) at (10,0.0) {$0$}; \node[draw,circle] (jj0) at (15.5,0.0) {$0$};
 \node[] (iidots) at (10,-0.7)   {$\vdots$};  \node[] (jjdots) at (15.5,-0.7)   {$\vdots$};
 \node[draw,circle] (iii) at (10,-1.7) {$i$}; \node[draw,circle] (jji) at (15.5,-1.7) {$i$};
 \node[] (iidots2) at (10,-2.4)   {$\vdots$};  \node[] (jjdots2) at (15.5,-2.4)   {$\vdots$};
 \node[draw,circle] (iim) at (10,-3.4) {$m$}; \node[draw,circle] (jjm) at (15.5,-3.4) {$m$};

 \node[] (jjdots3)  at (15.5,-4.1) {$\vdots$};
 \node[draw,circle] (jjj) at (15.5,-5.1) {$j$};
 \node[] (jjdots4)  at (15.5,-5.8) {$\vdots$};
 \node[draw,circle] (jjn) at (15.5,-6.8) {$n$};
 \draw[->] (ii0)--(jj0);
 \draw[->] (iii)--(jji);
 \draw[->] (iim)--(jjm);
 \draw[->] (ii0)--(jjn);
 \draw[->] (iii)--(jjn);
 \draw[->] (iim)--(jjn);
  \draw[->] (iim)--(jjj);
    \draw[->] (ii0)--(jjm);
    \draw[->] (ii0)--(jjj);

    \draw[->] (iii)--(jjm);
    \draw[->] (iii)--(jjj);

\end{tikzpicture}
\caption{\label{PQtransports} The optimal transport $z$ for $d_{m,n}$ when  $\alpha\in[\frac12,1]$ (left) and the feasible transport $\tilde z$ for $\cm_{m,n}$ (right). In both cases the demands $\pi_i^n$ at destinations $i\in \{0,\ldots,m\}$ are shipped at zero cost directly from their twin sources $i$  (horizontal arrows). The transports differ in how they manage the excess supplies $\pi_i^m-\pi_i^n$ at the sources, to meet the demands at  destinations $j\in\{m\!+\!1,\ldots,n\}$.
The feasible transport $\tilde z$ distributes the excesses proportionally to the demands at destination nodes. In contrast, $z$ satisfies all the demands $j\in\{m+1,\ldots,n-1\}$ from source $m$, and all the remaining residual supplies are shipped to the last destination $n$.} 
\end{figure}
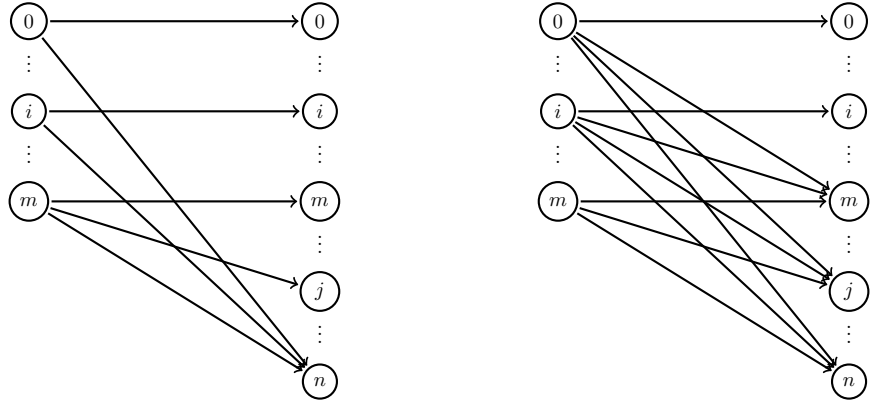

Noting that the $\cm_{m,n}$'s correspond to the feasible transport plan
\begin{equation}\tilde z_{i,j}=\left\{
\begin{array}{cl}
\pi_j^n&\text{for }i=j\in\{0,\ldots,m\}\\
\pi_i^m\pi^n_j&\text{for }i\in\{0,\ldots,m\} \text{ and } j\in\{m+1,\ldots,n\}
\\
0&\text{otherwise}
\end{array}
\right.
\end{equation}
it follows inductively that $d_{m,n}\leq \cm_{m,n}$, so the bounds $\cm_{m,n}$ need not be tight. This was already observed in \cite{bc18} for nonexpansive maps, although the difference between $\cm_{m,n}$ and $d_{m,n}$ was shown to vanish as $\alpha\to 1$. The plots in Figure \ref{fig6} suggest that the same phenomenon persists for all $\lip\in(0,1]$, with $\cm_{m,n}$ closely tracking the minimax-tight quantities $d_{m,n}$.
For $\alpha\in[\frac12,1]$, Proposition \ref{cmndmn} provides the following estimate for the absolute gap between the corresponding error bounds
$$0\leq \bm_m-\Rtight_m\leq (1-\alpha)^2(m+1)\lipalpha^{\hspace{0.3ex}m}/\alpha.$$
This is particularly informative when $\alpha$ is close to one, or whenever $(1-\alpha)^2(m+1)/\alpha$ is small. Actually, our numerical computations show substantially smaller gaps over a broad range of parameters, particularly for the optimal parameter $\alpha_\lip(m)$ (see Appendix \ref{AppE_Numerical}). We conjecture that the relative error $\bm_m/\Rtight_m$ is bounded by $\sqrt{3/2}\approx 1.23$
for all $\alpha\in[\frac12,1]$,  $\lip\in[\frac12,1]$, and $m\geq 1$. However, establishing a rigorous upper bound for  $\bm_m/\Rtight_m$ remains as an open question.

\begin{figure}[ht!]
  \centering
  \begin{minipage}{.3\linewidth}
    \centering
    \includegraphics[width=\linewidth]{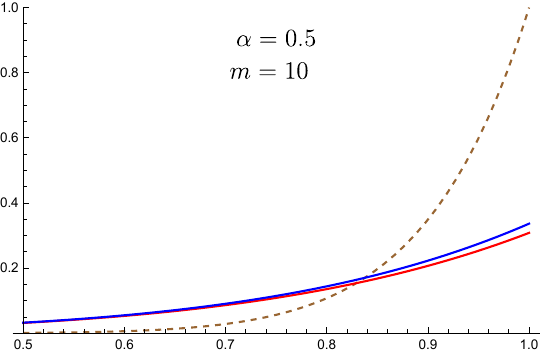}
  \end{minipage}
  \begin{minipage}{.3\linewidth}
    \centering
    \includegraphics[width=\linewidth]{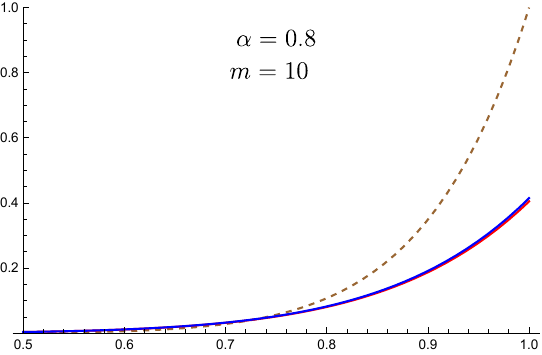}
  \end{minipage}
  \begin{minipage}{.3\linewidth}
    \centering
    \includegraphics[width=\linewidth]{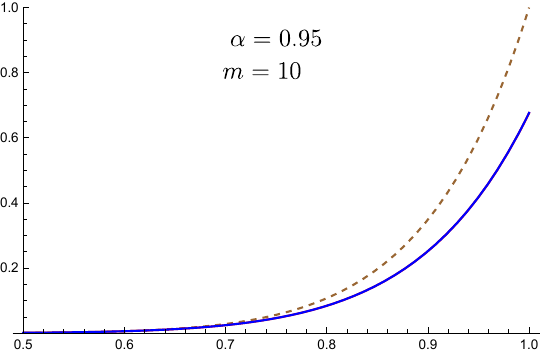}
  \end{minipage}

  \caption{Residual bounds after $m=10$ iterations as a function of $\lip\in[\frac12,1]$, for different relaxation parameters $\alpha$. The brown dashed curve is the normalized
  coarse   Banach--Picard bound $\lip^m\!$. The blue and red curves---barely distinguishable in the center and right panels---correspond, respectively, to the bounds $\bm_m$ and 
  $\Rtight_m$ for \eqref{eq:km_a}. All three curves coincide in the limit  $\alpha\uparrow 1$.}
  \label{fig6}
\end{figure}

\section{Probabilistic and integral representations of \texorpdfstring{$\cm_{m,n}$}{cmn}}
\label{sec:3}
Although the optimal transport bounds $\Rtight_m$ are minimax optimal and therefore
$\Rtight_m\leq \bm_m$, they are defined by an involved recursion and, to the best of our knowledge, no closed form formula is currently available. 
This is the major advantage of the suboptimal $\cm_{m,n}$'s and $\bm_m$'s, which have explicit
analytic expressions.
In this section we explore alternative probabilistic and  integral representations for the $\cm_{m,n}$'s, which provide more tractable expressions that will prove useful in the analysis of the $\bm_m$'s in the next section. 

\subsection{Representation by binomial distributions}\label{sec:3.1}
Let us consider two mutually independent i.i.d. sequences  of Bernoulli variables $B_i\sim\Ber(\alpha)$ and $D_i\sim\Ber(\palpha)$ with $\palpha\triangleq\alpha\,\lip/\lipalpha$, and the counting processes
$$X_n=B_1+\cdots+B_n\qquad; \qquad Y_m=D_1+\cdots+D_m.$$
 In order to derive a simpler expression for $\cm_{m,n}$ we exploit the following identity.

\begin{lemma}\label{L3}
For all $n,k\in\N$ we have
$\PP\!\bigl(X_n>k\bigr)=\alpha\sum_{j=0}^{n-1}\binom{n-1-j}{k}\alpha^k(1\!-\!\alpha)^{n-1-j-k}$.
\end{lemma}
\begin{proof}
With a change of index $l=n-1-j$ the sum on the right hand side becomes
$\sum_{l=0}^{n-1}\binom{l}{k}\alpha^{k+1}(1-\alpha)^{l-k}$.
The term $\binom{l}{k}\,\alpha^{k+1}(1-\alpha)^{l-k}$ is zero for $l<k$, and for $l\geq k$ is the probability of having
exactly $k$ successes among the first $l$ trials and a success at trial $l+1$.
Thus, denoting $T_{k+1}$ the trial of the $(k+1)$-st success the sum  above is equal to $\PP(T_{k+1}\leq n)$, and the conclusion follows from the standard relation
$\mathbb{P}(T_{k+1}\le n)=\mathbb{P}(X_n>k)$.
\end{proof}

\begin{proposition}\label{Prop4}
For  all $0\leq m\le n$ we have
$\cm_{m,n}
=
\lipalpha^{\hspace{0.3ex}m}\,\PP\bigl(X_n\!>\! Y_m\bigr)
-
\lipalpha^{\hspace{0.2ex}n}\,\PP\bigl(X_m\!>\!Y_n\bigr)$.
\end{proposition}

\begin{proof} The equations \eqref{eq:cmn}-\eqref{eq:latticepaths} can be written as $\cm_{m,n}=A_{m,n}-B_{m,n}$ with
 \begin{align*}
A_{m,n}&=\mbox{$\alpha\sum_{j=0}^{n-1} \sum_{k=0}^m
\binom{m}{k}\!\binom{n-1-j}{k}(\alpha^2\lip)^k(1\!-\!\alpha)^{n+m-1-j-2k},$}\\
B_{m,n}&=\mbox{$\alpha\sum_{j=0}^{n-1} \sum_{k=0}^m
\binom{n}{k}\!\binom{m-1-j}{k}(\alpha^2\lip)^k(1\!-\!\alpha)^{n+m-1-j-2k}$}.
\end{align*}
Exchanging the order of the sums in $A_{m,n}$ we can rewrite it as
 \begin{align*}
A_{m,n}&=\mbox{$\sum_{k=0}^m\binom{m}{k}(\alpha\lip)^k(1-\alpha)^{m-k} \left(\alpha\sum_{j=0}^{n-1} 
\!\binom{n-1-j}{k}\alpha^k(1\!-\!\alpha)^{n-1-j-k}\right),$}\\
&=\mbox{$\lipalpha^m\sum_{k=0}^m\binom{m}{k}\,\palpha^k(1-\palpha)^{m-k} \;\PP(X_n>k),$}
\end{align*}
where this last equality follows from Lemma \ref{L3}, together with $\alpha\lip=\palpha\,\lipalpha$ and $1-\alpha=
(1-\palpha)\,\lipalpha$. Then, by conditioning on the events $\{Y_m=k\}$, we obtain $A_{m,n}=\lipalpha^m\,\PP(X_n>Y_m)$.

Similarly, for $B_{m,n}$ we observe that $\binom{m-1-j}{k}$ is zero for $k=m$ and  also for $j\geq m$, so one can restrict the inner sum to $m-1$
and extend the outer sum to $n$. Then, a symmetric argument as above yields $B_{m,n}=\lipalpha^n\,\PP(X_m>Y_n)$, which combined with the expression for $A_{m,n}$ proves the result. 
\end{proof}

\subsection{Integral representations}\label{sec:trigo}
Starting from Proposition \ref{Prop4} and using Cauchy's residue theorem, we can represent the $\cm_{m,n}$'s by some specific trigonometric integrals. 

To this end, consider the curve 
$z_{\alpha}(\theta)=1-\alpha+\alpha\,\sigma\,e^{\ii\,\theta}$ in the complex plane. For $\theta\in(-\pi,\pi)$ write $z_{\alpha}(\theta)=r_{\!\alpha}(\theta)e^{\ii\,\varphi_{\!\alpha}(\theta)}$ in polar form 
with $r_{\!\alpha}(\theta)\triangleq|z_{\alpha}(\theta)|$ and $\varphi_{\!\alpha}(\theta)\triangleq\arg(z_{\alpha}(\theta))\in (-\pi,\pi)$ the principal argument, so that
 $r_{\!\alpha}(-\theta)=r_{\!\alpha}(\theta)$ and $\varphi_{\!\alpha}(-\theta)=-\varphi_{\!\alpha}(\theta)$.
 At $\theta=\pi$ we define $r_{\!\alpha}(\pi)$ and $\varphi_{\!\alpha}(\pi)$ by their left limits at $\theta\uparrow \pi$.

 When $\lip=1$, the formulas \eqref{eq:general-corrected}, \eqref{eq:cmmp1_Chebyshev0}, and \eqref{eq:cmmp1_Chebyshev22} are understood as improper integrals at the origin, the singularity being integrable, as established by dominated convergence in the proof.
 
\begin{proposition}\label{prop:int} 
With $r_{\!\alpha}(\theta)$ and $\varphi_{\!\alpha}(\theta)$ as above, for all $0\leq m\le n$ we have 
 \begin{equation}\label{eq:general-corrected}
\cm_{m,n}
=
\lipalpha^{\hspace{0.3ex}m}-\lipalpha^{\hspace{0.2ex}n}
+
\frac{2\sigma}{\pi}
\int_0^\pi
\frac{
r_{\!\alpha}(\theta)^{m+n}
\sin\!\bigl((n\!-\!m)\varphi_{\!\alpha}(\theta)\bigr)\sin\theta
}{
1-2\sigma\cos\theta+\sigma^2
}\,d\theta.
\end{equation}
\end{proposition}
\begin{proof} The case $n=m$ being trivial, we take $n>m$.
Let us assume first that $\lip\in (0,1)$ and set $g_d\triangleq
\lipalpha^{\hspace{0.3ex}m}\,\PP\bigl(X_n-Y_m=d\bigr)
-
\lipalpha^{\hspace{0.2ex}n}\,\PP\bigl(X_m-Y_n=d\bigr)$ so that $\cm_{m,n}=\sum_{d=1}^ng_d$.
Consider the  meromorphic complex function $G(z)=\sum_{d=-n}^ng_d\,(\sigma/z)^d$. Integrating over the unit circle, Cauchy's residue theorem gives
\begin{align}
\cm_{m,n}&=\frac{1}{2\pi \ii}\int_{|z|=1}\mbox{$
\frac{1}{\sigma}(\sum_{k=0}^{n-1}(z/\sigma)^k$})G(z)\,dz\nonumber\\
&=\frac{1}{2\pi \ii}\int_{|z|=1}\mbox{$
\frac{1-(z/\sigma)^n}{\sigma-z}$}G(z)\,dz\nonumber\\
&=\frac{1}{2\pi \ii}\int_{|z|=1}\mbox{$
\frac{G(z)}{\sigma-z}$}\,dz+\frac{1}{2\pi \ii}\int_{|z|=1}\mbox{$
\frac{(z/\sigma)^n}{z-\sigma}$}G(z)\,dz.\label{eq:integral}
\end{align}
The function $G(z)$ can be expressed in the form
\begin{align*}
G(z)&=\lipalpha^{\hspace{0.3ex}m}\,\EE\big[(\sigma/z)^{X_n-Y_m}\big]-\lipalpha^{\hspace{0.2ex}n}\,\EE\big[(\sigma/z)^{X_m-Y_n}\big]\\
&=\lipalpha^{\hspace{0.3ex}m}\,\EE\big[(\sigma/z)^{X_n}\big]\,\EE\big[(z/\sigma)^{Y_m}\big]-\lipalpha^{\hspace{0.2ex}n}\,\EE\big[(\sigma/z)^{X_m}\big]\,\EE\big[(z/\sigma)^{Y_n}\big]\\
&=\lipalpha^{\hspace{0.3ex}m}\big(1-\alpha+\alpha\,\sigma/z\big)^n\big(1-\palpha+\palpha\, z/\sigma\big)^m-\lipalpha^{\hspace{0.2ex}n}\,\big(1-\alpha+\alpha\,\sigma/z\big)^m\big(1-\palpha+\palpha\, z/\sigma\big)^n\\
&=\big(1-\alpha+\alpha\,\sigma/z\big)^n\big(1-\alpha+\alpha\,\sigma z\big)^m-\big(1-\alpha+\alpha\,\sigma/z\big)^m\big(1-\alpha+\alpha\,\sigma z\big)^n
\end{align*}
where the third line follows since the counting processes are sums of independent Bernoullis, and the fourth by using $\lipalpha(1\!-\!\palpha)=(1\!-\!\alpha)$ and $\lipalpha \,\palpha/\sigma=\alpha\,\sigma$.

Now, the function $H(z)=\mbox{$G(z)
{(z/\sigma)^n}/{(z-\sigma)}$}$ in the second integral in \eqref{eq:integral} has only one simple pole at $\sigma$ with residue $\mathop{\rm Res}(H,\sigma)=G(\sigma)
=\lipalpha^{\hspace{0.3ex}m}-\lipalpha^{\hspace{0.2ex}n}$
so that
\begin{align*}
\cm_{m,n}
&=\frac{1}{2\pi}\int_{-\pi}^{\pi}
\frac{G(e^{\ii\,\theta})}{\sigma-e^{\ii\,\theta}}e^{\ii\,\theta}\,d\theta+\big(\lipalpha^{\hspace{0.3ex}m}-\lipalpha^{\hspace{0.2ex}n}\big).
\end{align*}
Rationalizing the fraction in the first integral we have
$$\frac{e^{\ii\,\theta}}{\sigma-e^{\ii\,\theta}}=\frac{(\sigma e^{\ii\,\theta}-1)}{1-2\sigma\cos\theta+\sigma^2}=\frac{(\ii\,\sigma\sin\theta+\sigma\cos\theta-1)}{1-2\sigma\cos\theta+\sigma^2}.$$
On the other hand, $G(e^{i\theta})={\overline z}^n z^m-{\overline z}^m z^n$ with $z=z_{\alpha}(\theta)=r_{\!\alpha}(\theta)\,e^{\ii\,\varphi_{\!\alpha}(\theta)}$, so that
\begin{align*}
    G(e^{i\theta}) &=-2\ii \,r_{\!\alpha}(\theta)^{m+n}\sin\big((n\!-\!m)\varphi_{\!\alpha}(\theta)\big),
\end{align*}
and therefore 
\begin{align*}
\cm_{m,n}
&=\lipalpha^{\hspace{0.3ex}m}-\lipalpha^{\hspace{0.2ex}n}+\frac{1}{\pi}\int_{-\pi}^{\pi}
\frac{r_{\!\alpha}(\theta)^{m+n}\sin\big((n\!-\!m)\varphi_{\!\alpha}(\theta)\big)(\ii(1\!-\!\sigma \cos\theta)+\sigma\sin\theta)}{1-2\sigma\cos\theta+\sigma^2}
\,d\theta.
\end{align*}
The symmetries $r_{\!\alpha}(-\theta)=r_{\!\alpha}(\theta)$ and $\varphi_{\!\alpha}(-\theta)=-\varphi_{\!\alpha}(\theta)$
imply that the imaginary part in this integral vanishes, while the real part is twice the integral over $[0,\pi]$, which yields 
\eqref{eq:general-corrected}.

Now, the boundary case $\lip=1$ follows from $\lip<1$ since both sides of \eqref{eq:general-corrected} are continuous in $\lip$. This is clear for $\cm_{m,n}$ which is polynomial in $\lip$. The integral on the right is slightly more delicate because the denominator $(1-\sigma)^2+2\sigma(1-\cos\theta)$ vanishes at $\theta=0$ for $\sigma=1$ (recall that $\sigma=\sqrt\lip$). 
However, using $0\leq r_{\!\alpha}(\theta)\leq 1$ and $0\leq\varphi_{\!\alpha}(\theta)\leq \theta$ for $\theta\in (0,\pi)$, we have 
$$2\sigma\left|\frac{r_{\!\alpha}(\theta)^{m+n}\sin\big((n\!-\!m)\varphi_{\!\alpha}(\theta)\big)\sin\theta}{1-2\sigma\cos\theta+\sigma^2}\right|
\le\frac{(n-m)\,\theta\sin\theta}{1-\cos\theta}\leq 2(n-m),$$
so that by dominated convergence
 the integral is continuous at $\sigma=1$,
or equivalently at $\lip=1$.
\end{proof}

\noindent{\sc Remark.} Using the Chebyshev polynomials $U_k(x)$ and the identity
$\sin(k\,x)=U_{k-1}(\cos x)\,\sin x$,
we can rewrite \eqref{eq:general-corrected} in the alternative form (with the convention $U_{-1}(\cdot)\equiv 0$)
\begin{equation*}
\cm_{m,n}
=
\lipalpha^{\hspace{0.3ex}m}-\lipalpha^{\hspace{0.2ex}n}
+
\frac{2\alpha\lip}{\pi}
\int_0^\pi
\frac{
r_{\!\alpha}(\theta)^{m+n-1}
U_{n-m-1}\!\left((1-\alpha+\alpha\sigma\cos\theta)/{r_{\!\alpha}(\theta)}\right)
\sin^2\theta
}{
1-2\sigma\cos\theta+\sigma^2
}\,d\theta.
\end{equation*}
This representation is understood for $r_{\!\alpha}(\theta)>0$. At the exceptional point $\theta=\pi$ and $\alpha=1/(1+\sigma)$, the
integrand is defined by continuous extension. In fact, its value may be assigned arbitrarily at this single point without affecting the integral.

\section{Integral formulas and optimization of the error bounds \texorpdfstring{$\bm_m$}{bm}}
\label{sec:4}

Let us now focus on the contiguous terms
$\cm_{m,m+1}$, which provide the error bounds for the residuals $$\|x_m-Tx_m\|\leq \kappa\,\cm_{m,m+1}/\alpha=\kappa\,\bm_{m}.$$ 
We proceed to establish our second main result Theorem \ref{Prop6}. Recalling the auxiliary parameters
\begin{equation}\label{notation}
    \left\{\begin{array}{l}
    \salpha=(1-\alpha+\alpha\,\sigma)^2,\\
    ~\eta\,=4\alpha(1\!-\!\alpha)\sigma/\salpha,\\
    ~\xi\;=4\sigma/(1\!-\!\sigma)^2,
    \end{array}\right.
\end{equation}
the result states the following three characterizations  
\begin{equation}\label{eq:cmmp1_Chebyshev0}
\bm_{m}
=
(1-\lip)\,\lipalpha^{\hspace{0.3ex}m}
+
\frac{2\lip}{\pi}
\int_0^\pi
\frac{r_{\!\alpha}(\theta)^{2m}\sin^2\theta}{1-2\sigma\cos\theta+\sigma^2}\;d\theta,
\end{equation}
\begin{equation} \label{eq:cmmp1_Chebyshev22}
\bm_{m}=(1-\lip)\,\lipalpha^{\hspace{0.3ex}m}+
\left(\frac{8\lip}{\pi}
\int_0^1
\frac{\sqrt{x(1-x)}\;(1-\eta\,x)^{m}}{(1-\sigma)^2+4\sigma\,x}\,dx\right)\salpha^{\hspace{0.2ex}m},
\end{equation}
\begin{equation}\label{eq:cmmp1_Chebyshev55}
\bm_{m}
=\left\{\begin{array}{ll}
(1-\lip)\,\lipalpha^{\hspace{0.3ex}m}
+
\frac{\lip}{(1-\sigma)^2}\,\Fone\big(\mbox{$\frac{3}{2}$};-m,1;3;\eta,-\xi\big)\,\salpha^{\hspace{0.2ex}m}&\text{ if }\lip<1,\\[2ex]
{}\Gauss\!\left(\mbox{$\frac12$},-m;2;4\alpha(1\!-\!\alpha)\right)&\text{ if }\lip=1.
\end{array}\right.
\end{equation}

\subsection{Proof of Theorem \ref{Prop6}}
The first formula \eqref{eq:cmmp1_Chebyshev0} follows directly from \eqref{eq:general-corrected} by using the identity $\sin\varphi_{\!\alpha}(\theta)=\alpha\,\sigma\sin\theta/r_{\!\alpha}(\theta)$.
Next, \eqref{eq:cmmp1_Chebyshev22} follows with the change of variable $x=\sin^2(\theta/2)$, noting that 
 \begin{equation*}
 r_{\!\alpha}(\theta)^2=(1\!-\!\alpha)^2+2\alpha(1\!-\!\alpha)\sigma\cos\theta+\alpha^2\sigma^2=\big(1-\eta\,\sin^2(\theta/2)\big)\,\salpha,
 \end{equation*}
so that
$r_{\!\alpha}(\theta)^{2m}=(1-\eta\,x)^m\salpha^{\hspace{0.2ex}m}$,  $\sin^2\theta=4x(1-x)$, $\cos\theta=1-2x$, and $d\theta=dx/\sqrt{x(1\!-\!x)}$. 

Now, for $\lip<1$ we can factor the term $(1-\sigma)^2$ in the denominator of \eqref{eq:cmmp1_Chebyshev22}, and then  \eqref{eq:cmmp1_Chebyshev55} follows by matching the integral with
 Euler's  representation of $\Fone(a;b_1,b_2;c;u,v)$ (see \eqref{Eq:Euler_F1} in Appendix \ref{App:Hypergeom}) for $a=\frac32$, $b_1=-m$, $b_2=1$, $c=3$, $u=\eta$, $v=-\xi$, and $\frac{\Gamma(c)}{\Gamma(a)\Gamma(c-a)}=\frac{8}{\pi}$.
Similarly, when $\lip=1$ we have $\lipalpha=\salpha=\sigma=1$ so that \eqref{eq:cmmp1_Chebyshev22} becomes
\begin{align*}
\bm_{m}
&=
\frac{2\sigma}{\pi}\int_0^1 (1-\eta\,x)^mx^{-1/2}(1-x)^{1/2}\,dx,
\end{align*}
and then \eqref{eq:cmmp1_Chebyshev55} follows by matching this integral 
with Euler's representation \eqref{Eq:Euler_2F1} of ${}\Gauss(a,b;c;z)$  for 
$a=\frac12$, $b=-m$, $c=2$, $z=\eta=4\alpha(1\!-\!\alpha)$,
and $\frac{\Gamma(c)}{\Gamma(a)\Gamma(c-a)}=\frac{2}{\pi}$. \hfill$\Box$

\vspace{2ex}
Although \eqref{Eq:RmNuevo} gives a simpler $\Fone$ expression for $\bm_m$ that avoids the extra parameters $\salpha,\eta$ and $\xi$, an advantage of \eqref{eq:cmmp1_Chebyshev22} is that it readily implies the following weaker but simpler upper bounds. Notice that $\eta\leq 1$ since $\salpha-4\alpha(1-\alpha)\sigma=(1-\alpha-\alpha\sigma)^2\geq 0$.
\begin{corollary}\label{prop7}
For all $m\geq 0$ we have $\bm_{m}\le \lipalpha^{\hspace{0.3ex}m}$.
Moreover, with the notations of \emph{Theorem \ref{Prop6}}, for all $\alpha\in(0,1)$ we have the following three alternative upper bounds  (with \eqref{m321} interpreted as $\infty$ if $\lip=1$)
\begin{align}
\label{eq:bound}
\bm_{m}
&\le
(1-\lip)\,\lipalpha^{\hspace{0.3ex}m}
+
{\sigma }{}\Gauss(\mbox{$\frac12$},-m;2;\eta)\,\salpha^{\hspace{0.2ex}m},\\
\label{eq:bound2}
\bm_{m}&\leq
(1-\lip)\,\lipalpha^{\hspace{0.3ex}m}
+
\frac{2\sigma }{\sqrt{\pi\eta}}\frac{\Gamma(m+1)}{\Gamma(m+3/2)}\,\salpha^{\hspace{0.2ex}m},\\
\bm_m&
\leq (1-\lip)\lipalpha^{\hspace{0.3ex}m}+
\frac{4\lip}{\sqrt{\pi}\eta^{3/2}(1-\sigma)^2}
\frac{\Gamma(m+1)}{\Gamma(m+5/2)}\,\salpha^{\hspace{0.2ex}m}.\label{m321}
\end{align}
\end{corollary}
\begin{proof}
The bound $\bm_m\le\lipalpha^{\hspace{0.3ex}m}$ results from dropping the term $(1-\eta\,x)^m$ in the integral  \eqref{eq:cmmp1_Chebyshev22}, and using  $\frac{8}{\pi}\int_0^1
\frac{\sqrt{x}\,\sqrt{1-x}}{(1-\sigma)^2+4\sigma\,x}\,dx=1$ and  $\salpha\leq\lipalpha$.
If we instead drop $(1-\sigma)^2$ in the denominator we get 
\begin{align}\label{eq:otramas}
\bm_{m}
&\leq
(1-\lip)\,\lipalpha^{\hspace{0.3ex}m}
+
\frac{2\sigma \salpha^{\hspace{0.2ex}m}}{\pi}\int_0^1 (1-\eta\,x)^mx^{-1/2}(1-x)^{1/2}\,dx,
\end{align}
and \eqref{eq:bound} follows from Euler's integral representation of Gauss's hypergeometric $\!\Gauss$ (see \eqref{Eq:Euler_2F1}).
If we further remove the term $(1-x)^{1/2}$ in \eqref{eq:otramas} and
use the substitution \(y=\eta\, x\) we get
\begin{align*}
\bm_{m}
&\leq
(1-\lip)\,\lipalpha^{\hspace{0.3ex}m}
+
\frac{2\sigma \salpha^{\hspace{0.2ex}m}}{\pi\sqrt{\eta}}\int_0^\eta (1-y)^m y^{-1/2}\,d y.
\end{align*}
Since $\eta\leq 1$ this can be relaxed by extending the integration interval to $[0,1]$ which gives the Beta integral $
\mbox{$B\!\left(\frac12,m+1\right)$}
=
\sqrt{\pi}\,{\Gamma(m+1)}/{\Gamma(m+\frac32)}$, and
a direct substitution yields \eqref{eq:bound2}.
Similarly, by dropping the terms $\sqrt{1-x}$ and $4\sigma x$ in \eqref{eq:cmmp1_Chebyshev22}, and using the substitution $y=\eta\, x$, we get
\begin{align*}
\bm_m
&\le(1-\lip)\lipalpha^{\hspace{0.3ex}m}+
\mbox{$\frac{8\lip}{\pi(1-\sigma)^2\eta^{3/2}}\,\salpha^{\hspace{0.2ex}m}$}\!\!
\int_0^\eta\!\!\mbox{$
\sqrt{y}\,(1-y)^m\,dy$}.
\end{align*}
Extending the integral up to 1 we get the Beta integral 
$\frac12\sqrt{\pi}\,
\Gamma(m+1)/\Gamma(m+\frac{5}{2})$, which yields \eqref{m321}.
\end{proof}

\vspace{1ex}
\noindent{\sc Remarks.}\\[0.5ex]
{\em a)} For $\alpha,\lip\in (0,1)$ we have $\salpha<\lipalpha$ so that the first term $(1-\lip)\,\lipalpha^{\hspace{0.3ex}m}$ in all these upper bounds is  dominant for large $m$. However, as \(\lip\to 1\) this term vanishes and the second terms become relevant.\\[1ex]
{\em b)} Using the asymptotically tight inequalities $\frac{\Gamma(m+1)}{\Gamma(m+3/2)}\leq\frac{1}{\sqrt{m}}$ and
$\frac{\Gamma(m+1)}{\Gamma(m+5/2)}\leq \frac{1}{m^{3/2}}$ for $m\geq 1$ (see \citet{wendel}), the second terms in \eqref{eq:bound2} and \eqref{m321} behave as 
 $m^{-1/2}\salpha^{\hspace{0.2ex}m}$ and $m^{-3/2}\salpha^{\hspace{0.2ex}m}$ respectively.  
 While \eqref{m321} exhibits a faster decay,  the factor
$(1-\sigma)^2$ in the denominator degenerates as $\lip$ approaches 1.

\subsection{Optimizing the error bounds}
\label{sec:5}
Let us now consider a strict contraction with $\lip\in(0,1)$, and suppose we fix the number of iterations $m$. We seek to find 
$\alpha=\alpha_\lip(m)\in [0,1]$ that minimizes $\bm_m$, extending the polynomial $\bm_m$ continuously to $\alpha=0$.
Using the identity \eqref{eq:cmmp1_Chebyshev0} and denoting 
$$\sr(\alpha,\theta)\triangleq r_{\!\alpha}(\theta)^2=(1\!-\!\alpha)^2+2\alpha(1\!-\!\alpha)\sigma\cos\theta+\alpha^2\sigma^2,$$
the goal is to minimize 
\begin{equation}\label{eq:Phi-def}
\bm_{m}(\alpha)\triangleq
(1-\lip)\lipalpha^{\hspace{0.3ex}m}
+ \frac{2\lip}{\pi}\int_0^\pi
\frac{\sr(\alpha,\theta)^{m}\sin^2\theta}{1-2\sigma\cos\theta+\sigma^2}\, d \theta.
\end{equation}

\begin{proposition}\label{prop:strict-convex}
For each $m\ge 1$ and $\lip\in (0,1)$ the function $\alpha\mapsto\bm_{m}(\alpha)$ is strictly convex on $[0,1]$. Its unique minimizer $\alpha_\lip(m)\in[0,1]$ is such that $\alpha_\lip(m)=1$ if $\lip\le \tfrac12$ and  $\alpha_\lip(m)\in (\frac12,1)$ for $\lip>\tfrac12$.
\end{proposition}
\begin{proof}
Clearly $\alpha\mapsto(1-\lip)\lipalpha^{\hspace{0.3ex}m}$ is convex on $[0,1]$. Also, since $t\mapsto t^m$ is convex and strictly increasing on $[0,\infty)$,  its composition with the nonnegative strictly convex quadratic $\sr(\alpha,\theta)$  is strictly convex. This  implies that the integral in \eqref{eq:Phi-def} is strictly convex in $\alpha$ as well, and therefore the same holds for  $\alpha\mapsto\bm_{m}(\alpha)$. This ensures the existence of a unique minimizer $\alpha_\lip(m)\in [0,1]$.

We claim that $\bm'_{m}(\frac12)<0$ so that $\alpha_\lip(m)>\frac12$,
whereas $\bm'_{m}(1)=m\,\lip^{m-1}(2\lip-1)$ and therefore $\alpha_\lip(m)<1$ if and only if $\lip>\frac12$. To establish these claims, we differentiate under the integral sign 
\begin{equation*}\label{eq:Phi-prime}
\bm'_{m}(\alpha)
=-m(1-\lip)^2 \lipalpha^{\hspace{0.2ex}m-1}
+\frac{2m\lip}{\pi}\!\!\int_0^\pi
\frac{\sr(\alpha,\theta)^{m-1}\partial_\alpha\sr(\alpha,\theta)\sin^2\theta}{1-2\sigma\cos\theta+\sigma^2}\, d \theta,
\end{equation*}
which is justified by dominated
convergence: the denominator in \eqref{eq:Phi-def} is bounded below by $(1-\sigma)^2$, while $\sr(\alpha,\theta)$ and its $\alpha$-derivative are
uniformly bounded on $[0, 1]\times [0, \pi]$.

For $\alpha=\frac12$ we have 
$\partial_\alpha\sr(\frac12,\theta)=\lip-1<0$, hence  both terms in this sum are negative and $\bm'_{m}(\frac12)<0$.
At $\alpha=1$ one has $\lipalpha=\lip$, $\sr(1,\theta)=\lip$, 
$\partial_\alpha\sr(1,\theta)=2(\lip-\sigma\cos\theta)$,
so that
\begin{align*}
\bm'_{m}(1)
&=-m(1-\lip)^2\lip^{m-1}
+\frac{4m\lip^{m}}{\pi}
\int_0^\pi\frac{(\lip-\sigma\cos\theta)\sin^2\theta}{1-2\sigma\cos\theta+\sigma^2}\, d \theta,
\end{align*}
and using the classical identities
\begin{equation}\label{eq:I0}
\int_0^\pi \frac{\sin^2\theta}{1-2\sigma\cos\theta+\sigma^2}\, d \theta=\frac{\pi}{2}\quad;\quad
\int_0^\pi \frac{\cos\theta\,\sin^2\theta}{1-2\sigma\cos\theta+\sigma^2}\, d \theta=\frac{\pi\sigma}{4}
\end{equation}
we obtain
$\bm'_{m}(1)=m\,\lip^{m-1}(2\lip-1)$ as announced.
\end{proof}

\subsubsection{Asymptotics of $\alpha_\lip(m)$}
Focusing on the range $\lip\in (\frac12,1)$, since $\bm_{m}(\alpha)$ is polynomial in $\alpha$ it follows that $\alpha_\lip(m)$ is the unique root in $(0,1)$ of the
polynomial $\bm'_{m}(\alpha)$.
Although this root can be approximated numerically, 
below we analyze two asymptotic regimes of independent interest.

The first asymptotic is when the number of iterations $m$ is fixed and $\lip\approx 1$.
At $\lip=1$ the first term in \eqref{eq:Phi-def} disappears and 
$\sr(\alpha,\theta)=1-2\alpha(1\!-\!\alpha)(1-\cos\theta)$, hence $\bm_{m}(\alpha)$
only depends on $\alpha(1\!-\!\alpha)$ and its minimum is at $\alpha=1/2$. For $\lip\uparrow 1$, the uniform convergence  of the polynomials $\bm_m(\alpha)$ on $[0,1]$, together with uniqueness of the limiting minimizer, implies that
$\alpha_\lip(m)\longrightarrow \frac12
$ as $\lip\uparrow 1$.

The second asymptotic is when $\lip$ is fixed and the number of iterations $m$ is large. This regime is more delicate and shows a different picture. Our numerical experiments suggested $\alpha_\lip(m)\approx 1-c/m$ for some constant $c$ that depends on $\lip$. This can be made precise using the function
\begin{equation}\label{eq:profile-Psi}
\Psi_\lip(c)
\triangleq (1-\lip)\,e^{c\,(1-\lip)/\lip}
+\frac{2\lip}{\pi}\int_0^\pi
\frac{e^{2c\,(\cos\theta/\sigma-1)}\sin^2\theta}{1-2\sigma\cos\theta+\sigma^2}
\, d \theta.
\end{equation}

\begin{proposition}
For $\lip\in(\tfrac12,1)$ the function $\Psi_\lip(\cdot)$ has a unique minimizer $c^*_\lip\in[0,\infty)$, and we have
\begin{align}
\label{eq:min-profile}
\bm_{m}(\alpha_\lip(m))&=\lip^m\Psi_\lip(c^*_\lip)+ O(\lip^m/m),\\
\alpha_\lip(m)&=1-{c^*_\lip}/{m}+O(1/m^{3/2}),\label{eq:cm-def}
\end{align}
with asymptotic constants that depend on the fixed contraction factor $\lip$.
\end{proposition}

\begin{proof}Consider the strictly convex functions $f_m(c)\triangleq\lip^{-m}\,\bm_{m}(1-c/m)$ defined for $c\in[0,m]$. For $\alpha=1-c/m$ we have 
\begin{align}
\lip^{-m}(1-\lip)\,\lipalpha^{\hspace{0.2ex}m}&=(1-\lip)\left(1+\frac{c\,(1-\lip)/\lip}{m}\right)^m,\label{c1}\\
\lip^{-m}\sr(\alpha,\theta)^m&
=\left(1+\frac{2c\,(\cos\theta/\sigma-1)}{m}+\frac{c^2(1-2\sigma\cos\theta+\sigma^2)/\lip}{m^2}\right)^m,\label{c2}
\end{align}
from which it follows that $f_m(c)$ converges pointwise towards $\Psi_\lip(c)$ for all $c\geq 0$.
Notice that $\Psi_\lip(\cdot)$ is strictly convex since this holds for 
 $c\mapsto (1-\lip)e^{c(1-\lip)/\lip}$ while the integral term is convex.

Ignoring the integral term in $f_{m}(\cdot)$ (which is positive) and considering only \eqref{c1}, the inequality  $(1+\frac{x}{m})^m\geq q(x)\triangleq 1+x+\frac{1}{4}x^2$ for $x\geq 0$
 and $m\geq 2$, implies that the functions $f_m(\cdot)$ for $m\geq 2$, as well as the limit $\Psi_\lip(\cdot)$, are uniformly bounded from below by the strongly convex quadratic $q_\lip(c)\triangleq (1-\lip)\,q(c(1-\lip)/\lip)$.
On the other hand, for $\alpha=1$ we have $\lipalpha=\lip$ and $\sr(1,\theta)=\lip$, so that
$$
f_m(0)=\lip^{-m}\,\bm_{m}(1)=
(1-\lip)
+ \frac{2\lip}{\pi}\int_0^\pi
\frac{\sin^2\theta}{1-2\sigma\cos\theta+\sigma^2}\, d \theta = 1,
$$
and therefore $\Psi_\lip(0)=1$ as well. It follows that the unique minimizers $c^*_{\lip}$ of $\Psi_\lip(\cdot)$ and $c_{\lip}(m)$ of $f_m(\cdot)$ for $m\geq 2$, are all located in the interval $[0,\bar c]$ with $\bar c=\frac{2\lip}{1-\lip}(\frac{1}{\sqrt{1-\lip}}-1)$ the positive root of $q_\lip(\bar c)=1$. 
 
As a matter of fact, differentiating \eqref{eq:profile-Psi} and using the identities \eqref{eq:I0} we get $\Psi_{\lip}'(0)=(1-2\lip)/\lip<0$ and therefore $c^*_{\lip}>0$. Also, the minimizer of $f_m(\cdot)$ is precisely $c_{\lip}(m)=m(1-\alpha_\lip(m))$ which is strictly positive by Proposition \ref{prop:strict-convex}, and we have  $\bm_{m}(\alpha_\lip(m))=\lip^m f_m(c_{\lip}(m))$.

The formulas \eqref{c1}-\eqref{c2} involve expressions of the 
form $(1+\frac{x}{m})^m$ with $x$ confined to a bounded interval $[-a,a]$ (that depends on $\lip$). Since $|(1+\frac{x}{m})^m-e^x|\leq a^2 e^{2a}/m$ for all $m\geq 2a$ and $x\in[-a,a]$, we can find a constant $M$ such that 
$|f_m(c)-\Psi_\lip(c)|\leq M/m$ for all $m\geq M$ and $c\in[0,\bar c]$.
Evaluating at $c_{\lip}(m)$ and $c^*_\lip$, and  using their  optimality, we get
\begin{align*}
   f_m(c_{\lip}(m))&\le f_m(c^*_\lip) \le \Psi_\lip(c^*_\lip)+\mbox{$\frac{M}{m}$},\\
    \Psi_\lip(c^*_\lip)&\le \Psi_\lip(c_{\lip}(m))\le f_m(c_{\lip}(m))+\mbox{$\frac{M}{m}$}.
\end{align*}
It follows that $|f_m(c_{\lip}(m))-\Psi_\lip(c^*_\lip)|\le M/m$, which implies \eqref{eq:min-profile}. Also, 
$0\leq\Psi_\lip(c_{\lip}(m))-\Psi_\lip(c^*_\lip)\leq\frac{2M}{m}$ and, since $\Psi_\lip(\cdot)$ is strongly convex on $[0,\bar c]$, we obtain $|c_{\lip}(m)-c^*_\lip|=O(1/\sqrt{m})$ which yields \eqref{eq:cm-def}.
\end{proof}

\begin{proposition}\label{p12}
    As $\lip\uparrow 1$ we have $c_{\lip}^*(1-\lip)/\lip\to \frac12$, so that $c_{\lip}^*\sim \frac{\lip}{2(1-\lip)}$.
\end{proposition}
\begin{proof}
Denoting $y^*_{\lip}\triangleq c^*_{\lip}(1-\lip)/\lip$ we must show that $y^*_{\lip}\to\frac12$ as $\lip\uparrow 1$. 
To this end we observe that $y^*_{\lip}$
is the unique minimizer of  the rescaled function
$\Phi_\lip(y)\triangleq \frac{1}{\sqrt{1-\lip}}\Psi_\lip(y\,\lip/(1-\lip))$.

Recalling that  $\sigma=\sqrt{\lip}$ and using the change of variable $x=\sin^2(\theta/2)$ in \eqref{eq:profile-Psi}, we can write
    $$\Phi_\lip(y)=\sqrt{1-\lip}\,e^{y}
+\frac{8\lip}{\pi\sqrt{1-\lip}}e^{2y\sigma/(1+\sigma)}\int_0^1
\frac{e^{-4y\sigma\, x/(1-\lip)}\sqrt{x(1-x)}}{(1-\sigma)^2+4\sigma x}
\, d x,$$
and then a second change of variable $u=x/(1-\lip)$ gives
$$\Phi_\lip(y)=\sqrt{1-\lip}\,e^y+\frac{8\lip}{\pi}e^{2y\sigma/(1+\sigma)}\int_0^{1/(1-\lip)}
\frac{e^{-4y\sigma\,u}\sqrt{ u(1-(1-\lip)u)}}{(1-\sigma)/(1+\sigma)+4\sigma  u}
\, d u.$$
For $\lip\geq\frac12$ we have $\sigma \ge 1/\sqrt{2}$ so the integrand is bounded by $\exp(-2 y u)/\sqrt{2u}$. For $y>0$ the latter is integrable on $(0,\infty)$, so by dominated convergence we get $\Phi_\lip(y)\to \Phi_1(y)\triangleq e^y/\sqrt{\pi y}$ as $\lip\uparrow 1$.

The function $\Phi_1(\cdot)$ is strictly convex and diverges to  infinity at 0 and $\infty$, with its minimum attained at $y=\frac12$. Thus,
for $0<a<\frac12<b$ we have $\Phi_1(a)>\Phi_1(\frac12)$ and $\Phi_1(b)>\Phi_1(\frac12)$, and by pointwise convergence we also have $\Phi_\lip(a)>\Phi_\lip(\frac12)$ and $\Phi_\lip(b)>\Phi_\lip(\frac12)$ for all $\lip$ close to 1. Hence, since $\Phi_\lip(\cdot)$ is convex its minimizer $y^*_{\lip}$ is trapped in $[a,b]$, and we conclude  by letting $a\uparrow\frac12$ and $b\downarrow\frac12$.
\end{proof}

\section{Inexact Krasnosel'skii--Mann  iterations}\label{sec:6}
In this final section we analyze the case of inexact Krasnosel'skii--Mann iterations in which  the operator values can only be computed up to an additive error $e_m$, that is 
\begin{equation*} \tag{$\textsc{ikm}$} \label{eq:ikm_a}
 (\forall m\in\N)\qquad   x_{m+1}=(1-\alpha)\,x_{m}+\alpha\,(Tx_m+e_m).
\end{equation*}
We implicitly assume that $y_m\triangleq Tx_m+e_m\in C$, so  the iterates remain in the domain $C$. 
The nonexpansive case was studied in \citet{BravoCominettiPavezSigne2019}, so throughout this section we assume  $0<\lip<1$.

\vspace{1ex}
\noindent{\sc Remark.} To bypass the restrictive assumption  $y_m\in C$,   in Hilbert spaces with $C$  closed and convex, one can consider the projected iteration $x_{m+1}=(1-\alpha)\,x_m+\alpha\,\tilde y_m$ with $\tilde y_m=\Pi_C(y_m)$. This can be written as
$x_{m+1}=(1-\alpha)\,x_m+\alpha\,(Tx_m+\tilde e_m)$ with $\tilde e_m\triangleq\tilde y_m-Tx_m$ so that $Tx_m+\tilde e_m=\tilde y_m\in C$. Moreover, since $Tx_m\in C$ and $\Pi_C(\cdot)$ is nonexpansive, we have 
$\|\tilde e_m\|\le\|e_m\|$. Beyond the Hilbert setting, one may consider an approximate projection $\tilde y_m\in C$  with 
$\|\tilde y_m-y_m\|\leq \mathop{\rm dist}(y_m,C)+\tilde\varepsilon_m$. Since $\mathop{\rm dist}(y_m,C)\leq \|y_m-Tx_m\|=\|e_m\|$, a triangle inequality then gives 
$\|\tilde e_m\|\leq 2\|e_m\|+\tilde\varepsilon_m$.  Slight variations of these ideas were previously developed in \citet{BravoCominettiPavezSigne2019}.

\begin{proposition}
Let  $\varepsilon_m\geq 0$ such that $\|e_m\|\leq \kappa\,\varepsilon_m$. Set $\rew_i=\varepsilon_i/\lip$ and define inductively
$$\tilde \cm_{m,n}=\mbox{$\pi_0^m-\pi_0^n + \sum_{i=1}^m\sum_{j=m+1}^n\lip\,\pi_i^m \pi_j^n(\tilde \cm_{i-1,j-1}+\rew_{i-1}+\rew_{j-1})$.}$$
Then, for all $0\leq m<n$ we have $\|x_m-x_n\|\leq\kappa\,\tilde \cm_{m,n}$.
\end{proposition}
\begin{proof}
Denoting $y_{-1}=x_0$ we have $x_m=\sum_{i=0}^m\pi_i^m\,y_{i-1}$ for all $m\in\N$. The proof follows inductively. Suppose that $\|x_i-x_j\|\leq\kappa\,\tilde \cm_{i,j}$ for all $0\leq i<j\le n-1$. Then, for 
$0\leq m<n$ we get
 \begin{align*}
    \|x_m-x_n\|&=\|\mbox{$\sum_{i=0}^m\pi_i^m\,y_{i-1}-\sum_{j=0}^n\pi_j^n\,y_{j-1}$}\|\\
     &=\|\mbox{$\sum_{i=0}^m\sum_{j=m+1}^n\pi_i^m \pi_j^n(y_{i-1}-y_{j-1})$}\|\\
&\leq\kappa\,\mbox{$\pi_0^m \sum_{j=m+1}^n\pi_j^n + \kappa \sum_{i=1}^m\sum_{j=m+1}^n\pi_i^m \pi_j^n\,(\lip\tilde \cm_{i-1,j-1}+\varepsilon_{i-1}+\varepsilon_{j-1})$}\\
&=\kappa\left(\mbox{$\pi_0^m -\pi_0^n + \sum_{i=1}^m\sum_{j=m+1}^n\lip\,\pi_i^m \pi_j^n(\tilde \cm_{i-1,j-1}+\rew_{i-1}+\rew_{j-1})$}\right)\\
&=\kappa\,\tilde \cm_{m,n}.
\end{align*}

\vspace{-4ex}
\end{proof}
The bounds $\tilde \cm_{m,n}$ can be interpreted as the total expected reward
for the Markov chain in \S{}\ref{sec:2}, where now at each transient state $(i,j)$ with $0\leq i< j$   the horizontal player $H$ collects a reward $\rew_i$ and the vertical player $V$ gets $\rew_j$,  with  an additional  reward of 1 at every $H$-winning position $(-1,j)$.
Notice that, starting from $(m,n)$ with $m<n$, the chain can only visit transient 
states $(i,j)$ such that $0\leq i<m$ and $i< j< n$. Adapting the arguments in Sections \ref{sec:stratification} and \ref{sec:latticepaths}, and noting that the final jump  to $(i,j)$ involves an additional factor $\alpha\lip$ that 
is absent in a jump towards $(-1,j)$, the probability of visiting $(i,j)$ when starting from $(m,n)$ is (for alternative formulas see Appendix \ref{App:Alt_pmnij})
\begin{align}
p_{m,n}^{i,j}&=\sum_{k\geq 0}L_{m,n}^{k,i,j}\;(\lip\,\alpha^2)^{k+1}\,(1\!-\!\alpha)^{n+m-i-j-2k-2},\label{pmnij1}\\
L_{m,n}^{k,i,j}&=\binom{n\!-\!1\!-\!j}{k} \!\!\binom{m\!-\!1\!-\!i}{k}-\binom{n\!-\!1\!-\!i}{k}\!\! \binom{m\!-\!1\!-\!j}{k}.\label{pmnij2}
\end{align}
The sum in \eqref{pmnij1} can be restricted to $k<\min\{n-j,m-i\}$ since for larger $k$ one has 
$L_{m,n}^{k,i,j}=0$. This also avoids negative powers of $(1-\alpha)$ that are undefined when $\alpha=1$.

The expected reward $\tilde \cm_{m,n}$ starting from $(m,n)$ decomposes as $\tilde \cm_{m,n}=c^W_{m,n}+c^H_{m,n}+c^V_{m,n}$ where
$c^W_{m,n}$ is the terminal reward at states $(-1,j)$, and $c^H_{m,n}, c^V_{m,n}$  the expected transient rewards, that is
\begin{align*}
c^W_{m,n}&= \mathbb P \left ( \text{\normalfont $H$ wins $|\,s_0=(m,n)$} \right )=\cm_{m,n},\\
c^H_{m,n}&=\mbox{$\sum_{i=0}^{m-1}\PP(H_{m,n}^i)\,\rew_i$},\\
c^V_{m,n}&=\mbox{$\sum_{j=1}^{n-1}\PP(V_{m,n}^j)\,\rew_j$},
\end{align*}
with $\PP(H_{m,n}^i)=\sum_{j=i+1}^{n-1}p_{m,n}^{i,j}$ 
the probability of visiting column $i$,
and $\PP(V_{m,n}^j)=\sum_{i=0}^{j-1}p_{m,n}^{i,j}$  
the probability of visiting row $j$. Notice that the sum 
$c^V_{m,n}$ starts from $j=1$ because $\PP(V_{m,n}^0)=0$. 

Building on the results for absorption probabilities in Section \ref{sec:3}, we obtain the following characterization for the probability of $H^i_{m,n}$.
\begin{lemma}\label{L4}
   For all $(m,n)\in \D$ and $0\le i< m$ we have $\PP(H_{m,n}^i)= \alpha\lip\, \cm_{m-1-i,n-1-i}$.
\end{lemma}
\begin{proof}
The event $H^i_{m,n}$ is the same as the $H$-winning event in Lemma \ref{lem:interp2} with column $i$ instead of $-1$, the only difference being the extra factor $\alpha\lip$  in the final  jump to $(i,j)$. Hence, a shift in the indices establishes the stated formula for  $\PP(H_{m,n}^i)$.
\end{proof}

Appendices \ref{App:ExpectReward} and \ref{App:Alt_pmnij} provide formulas to estimate $\PP(V_{m,n}^j)$. Specifically, Proposition \ref{P11} shows that in the contiguous case $n=m+1$ we have $\PP(V_{m,m+1}^j)\leq\PP(H_{m,m+1}^j)$ for all $0\le j\leq m-1$.
 Exploiting these facts we derive the following error bound for the residuals of \eqref{eq:ikm_a}.

\begin{theorem}\label{Thm12} Assume $0<\lip<1$ and suppose that $Tx_m+e_m\in C$ so  
that the sequence  $(x_m)_{m\in\N}$ generated by \emph{\eqref{eq:ikm_a}} is well defined and remains in the domain $C$. Suppose that $\|e_m\|\leq \kappa\,\varepsilon_m$ and let $\bm_m$ be as in \eqref{eq:cmmp1_Chebyshev55} (or equivalently \eqref{eq:cmmp1_Chebyshev5}). Then, for all $m\in\N$ we have 
\begin{align}\label{eq:conv}
    \|x_m-Tx_m\|
    &\leq\kappa\,\big(\mbox{$\bm_{m}+ 2\alpha\sum_{i=0}^{m-1}\!\bm_{m-1-i}\,\varepsilon_{i}+2\varepsilon_m$}\big).
\end{align}
\end{theorem}
\begin{proof} For $m=0$ this holds trivially since $\|x_0-Tx_0\|\leq\kappa$ and $\bm_0=1$, so let us assume $m\geq 1$. Using \eqref{eq:ikm_a} and the error estimate $\|e_m\|\leq\kappa\,\varepsilon_m$, a triangle inequality gives
\begin{align}\label{ikmbound}
    \|x_m-Tx_m\|&\leq\|x_{m+1}-x_m\|/\alpha+\|e_m\|\leq\kappa\,(\tilde \cm_{m,m+1}/\alpha+\varepsilon_m),
\end{align}
with $\tilde \cm_{m,m+1}$  the total expected reward starting from $(m,m+1)$, which is given by 
\begin{align*}
\tilde \cm_{m,m+1}&=\mbox{$\cm_{m,m+1}+\sum_{i=0}^{m-1}\PP(H_{m,m+1}^{i})\,\rew_i+\sum_{j=1}^{m}\PP(V_{m,m+1}^{j})\,\rew_j$}
\end{align*}
where $\rew_i=\varepsilon_i/\lip$.
Using 
$\cm_{k,k+1}=\alpha\,\bm_{k}$ and Lemma \ref{L4}, we get $\PP(V^j_{m,m+1})\leq\PP(H_{m,m+1}^{j})=\alpha^2\lip \bm_{m-1-j}$ for all $0\le j<m$. This, together with $\PP(V_{m,m+1}^{m})=\alpha\lip(1-(1-\alpha)^m)\leq\alpha\lip$, implies
\begin{align*}
\tilde \cm_{m,m+1}
&\le\mbox{$\alpha\,\bm_{m}+\alpha^2\bm_{m-1}\,\varepsilon_0+ \sum_{i=1}^{m-1}2\alpha^2\,\bm_{m-1-i}\,\varepsilon_i+\alpha\,\varepsilon_m$}.
\end{align*}
The term $\alpha^2\,\bm_{m-1}\,\varepsilon_{0}$ can be included as term $i=0$ in the convolution sum (multiplied by 2), and then plugging this into \eqref{ikmbound} yields \eqref{eq:conv}.
\end{proof}

\noindent{\sc Remark.} 
The estimate \eqref{eq:conv} can be  sharpened by retaining the $i=0$ term  as $\alpha \bm_{m-1}\varepsilon_0$. For simplicity, we absorbed this term in the sum  multiplied by 2, which gives a slightly weaker bound.

\vspace{2ex}
To illustrate the use of Theorem \ref{Thm12}, we consider below perturbations $e_n$ with  a power law decay. 
\begin{proposition}
Assume $0<\lip<1$ and suppose that $\|e_m\|\leq \kappa\,\varepsilon_m$ for $\varepsilon_m=b/(m+1)^{a}$ with $a,b>0$. Let $\nu=1+\alpha\gamma/\delta^{a+1}$ with $\gamma=2^a e\,(1+a\,\Gamma(a))$ and $\delta=\alpha(1-\lip)$. Then, for  $m\geq 1$ we have
$$\|x_m-Tx_m\|\leq\mbox{$\kappa(\bm_m+ 2b\,\nu/m^a)$}.$$
\end{proposition}
\begin{proof}
With the change of index $k=m-1-i$ in \eqref{eq:conv} and denoting $y_m\triangleq \sum_{k=0}^{m-1} \bm_{k}\big(\frac{m}{m-k}\big)^{a}$, it suffices to show that
  $y_m\leq\gamma/\delta^{a+1}$.
From Corollary \ref{prop7} we have $\bm_k\le \lipalpha^{\hspace{0.2ex}k}\leq e^{-\delta k}$, which together with $\frac{m}{m-k}\le k+1$ for $0\leq k\leq m-1$, implies
$$y_m\le\mbox{$\sum_{k=0}^{\infty}e^{-\delta k}(k+1)^a$}.$$
Now, for $x\in[k,k+1]$ we have $e^{-\delta k} \le e\,e^{-\delta x}$ and $(k+1)^a\leq 2^a(1+x^a)$, so that
$$y_m
\le 2^ae\int_0^\infty \!\!\! e^{-\delta x}(1+x^a) dx=2^a e\left(1/\delta+\Gamma(a\!+\!1)/\delta^{a+1}\right).$$
Since $\delta<1$, we conclude
$y_m\leq 2^a e\,(1+a\,\Gamma(a))/\delta^{a+1}\!=\gamma/\delta^{a+1}$ as was to be proved.
\end{proof}

\noindent{\bf Acknowledgements.} The authors are grateful for the support provided by FONDECYT grant No. 1241805.

\vspace{2ex}
\noindent{\bf Disclosure.} This paper builds on several ideas developed in our previous work on Krasnosel'ski\u{\i}--Mann iterations for nonexpansive maps
\cite{csv14,bc18,bravo2024stochastic,BravoCominettiPavezSigne2019,bravo2022universal,bravo_cominetti_lee2026halpern,contreras2023optimal}, in which some connections with Markov chains had already been exploited. The key formula \eqref{eq:cmn} was initially discovered through experimentation: we used {\em Mathematica} to generate the polynomials arising from the recursion \eqref{eq:recursive} and then submitted their coefficients to the {\em On-Line Encyclopedia of Integer Sequences} (OEIS). This search suggested a connection with Dyck paths and ultimately led us to Krattenthaler's work on lattice paths \cite{krattenthaler2015}, which paved the way for the precise formulation and proof of Theorem \ref{thm:main}.
The connection with binomial distributions presented in \S\ref{sec:3.1} was identified later with the assistance of {\em ChatGPT}. The proof of Proposition \ref{Prop4}, however, as well as all subsequent analysis, is entirely our own. We take full responsibility for the content of the manuscript.

\appendix

\section{Basic facts on hypergeometric functions}\label{App:Hypergeom}

Let us recall Gauss's and Appell's hypergeometric functions which are given by the series
\begin{align}
    \Gauss(a,b;c;z)&\triangleq \sum_{k=0}^\infty\frac{(a)_k(b)_k}{(c)_k}\frac{z^k}{k!},\label{2F1_series}\\
    \Fone(a;b_1,b_2;c;u,v)&\triangleq \sum_{k=0}^\infty\sum_{l=0}^\infty\frac{(a)_{k+l}(b_1)_k(b_2)_l}{(c)_{k+l}}\frac{u^k}{k!}\frac{v^l}{l!},\label{F1_series}
\end{align}
with $(x)_k=\prod_{i=0}^{k-1}(x+i)$ the rising Pochhammer symbols. 

In general, $a,b,b_1,b_2,c,z,u,v$ are complex numbers and the convergence of the series is assured  for $|z|<1$, and $|u|<1$ and $|v|<1$ respectively.  The  $\Gauss$ series terminates when either $a$ or $b$ are nonpositive integers, which gives a well-defined polynomial for all $z\in\mathbb{C}$. For $\Fone$, the double series terminates  when $a$ is a nonpositive integer, and also if both $b_1$ and $b_2$ are nonpositive integers. Outside their disks of convergence, $\Fone$ and $\Gauss$ are understood by analytic continuation.

In the formulas  \eqref{eq:cmmp1_Chebyshev5} and \eqref{eq:cmmp1_Chebyshev55} the relevant continuation is fixed unambiguously by Euler’s integrals.
We recall that, for real parameters with $c>a>0$, Euler’s integral representations are
\begin{align}
{}\Gauss(a,b;c;z)&=\frac{\Gamma(c)}{\Gamma(a)\Gamma(c-a)}\int_0^1 \frac{x^{a-1}(1-x)^{c-a-1}}{(1-z\,x)^{b}}\;dx,\label{Eq:Euler_2F1}\\
\Fone(a;b_1,b_2;c;u,v)&=
\frac{\Gamma(c)}{\Gamma(a)\Gamma(c-a)}
\int_0^1
\frac{x^{a-1}(1-x)^{c-a-1}}{(1-u\, x)^{b_1}(1-v\, x)^{b_2}}\;dx,\label{Eq:Euler_F1}
\end{align}
which hold if the denominators stay positive on $[0,1]$:  $z<1$ in \eqref{Eq:Euler_2F1}, and $u<1$ and $v<1$ in \eqref{Eq:Euler_F1}. The endpoint $z=1$ in \eqref{Eq:Euler_2F1} requires the additional  integrability condition $c-a-b>0$, with similar conditions for \eqref{Eq:Euler_F1} when $u=1$ and/or $v=1$.
For complex parameters, further branch and argument conditions apply.

\vspace{1ex}
While there exist many  relations between $\Fone$ and $\Gauss$, the next lemma presents a specific identity, used in proving Proposition \ref{Thm:Alt_cmn},  which we could not find  in classical references.

\begin{lemma}\label{Le:Sum_2F1}
For all $a\in\N$ and $b,z,t\in\mathbb{C}$ we have
\begin{equation}\label{eq:main-identity}
\sum_{j=0}^{a} {}\Gauss(j-a,b;1;z)\;t^j
=
(a+1)\Fone(-a;b,1;2;z,1-t).
\end{equation}
\end{lemma}

\begin{proof}
Since $j-a$ in the sum on the left of \eqref{eq:main-identity} is a nonpositive integer, the Gauss hypergeometric series  is a finite sum and then, using 
the equality $(j-a)_k/k!=(-1)^k\binom{a-j}{k}$, we get
\begin{align*}
\sum_{j=0}^{a} {}\Gauss(j-a,b;1;z)\,t^j
&=\sum_{j=0}^{a} \mbox{$\left(\sum_{k=0}^{a-j}\frac{(j-a)_k(b)_k}{k!}\frac{z^k}{k!}\right)$}t^j
=
\sum_{k=0}^{a}\mbox{$ (b)_k\frac{(-z)^k}{k!}\left(\sum_{j=0}^{a-k}
\binom{a-j}{k}t^j\right)$}.
\end{align*}
Now, expanding the right hand side of \eqref{eq:main-identity}, and using the identities $(1)_l=l!$, $(2)_{k+l}=(k+l+1)!$, and
$(a+1)\,(-a)_{k+l}/(k+l+1)!=(-1)^{k+l}\binom{a+1}{k+l+1}$, after rearranging terms we have
\begin{align*}
(a+1)\Fone(-a;b,1;2;z,1-t)
&=\sum_{k=0}^a\mbox{$(b)_k\frac{(-z)^k}{k!}\left(\sum_{l=0}^{a-k}\binom{a+1}{k+l+1}(t-1)^l\right)$}.
\end{align*}
Thus, it remains to prove that
$$
\mbox{$\sum_{j=0}^{a-k}\binom{a-j}{k}t^j
=
\sum_{l=0}^{a-k}
\binom{a+1}{k+l+1}(t-1)^l.$}
$$
Denoting $P(t)$ the polynomial on the left, its Taylor expansion at $t=1$ is
\begin{equation*}
    P(t)=\mbox{$\sum_{l=0}^{a-k}\frac{1}{l!}P^{(l)}(1)(t-1)^l$}\quad\text{with}\quad
\mbox{$\frac{1}{l!}P^{(l)}(1)=
\sum_{j=l}^{a-k}\binom{j}{l}\binom{a-j}{k}$}.
\end{equation*}
The conclusion follows since Vandermonde's identity
shows that the latter sum is exactly 
$\binom{a+1}{k+l+1}$.
\end{proof}

\section{Probability estimates for expected rewards}\label{App:ExpectReward}

From Lemma \ref{L4} we have $\PP(H_{m,n}^i)=\alpha\lip\,\cm_{m-1-i,n-1-i}$ for $0\leq i<m<n$, which is equivalent to
\begin{equation}\label{pdoti}
\PP(H_{m,n}^i)=\alpha\lip\,\left[ \lipalpha^{\hspace{0.2ex}m-1-i}\,\PP(X_{n-1-i}>Y_{m-1-i})-\lipalpha^{\hspace{0.2ex}n-1-i}\,\PP(X_{m-1-i}>Y_{n-1-i})\right].
\end{equation}
where $(X_n)_{n\geq 1}$ and $(Y_m)_{m\geq 1}$ are the Bernoulli counting processes from Section \ref{sec:3.1}.
Adopting  the (asymmetric) conventions $X_l\equiv-\infty$ and $Y_l\equiv+\infty$ for $l<0$,
this formula is also valid for $i\geq m$.

As in Section \ref{sec:3} we could derive explicit formulas for $\PP(V^j_{m,n})$, involving more complicated combinations of hypergeometric functions. However, in the contiguous case $(m,m+1)$ we can show that $\PP(V^j_{m,m+1})$
is bounded above by $\PP(H^j_{m,m+1})$, which is enough for our purposes.
To this end, we first establish two auxiliary lemmas.

\begin{lemma}\label{L5}
For all $m,j,k\in \mathbb{N}$ we have
\begin{equation}\label{eq:ultima}
\mbox{$\alpha \sum_{i=0}^{j-1} \binom{m-1-i}{k}\alpha^k(1-\alpha)^{m-1-i-k}$}=
\PP(X_m> k)-\PP(X_{m-j}>k).
\end{equation}
\end{lemma}
\begin{proof} For $k>m-1-i$ we have $\binom{m-1-i}{k}=0$ and the corresponding summand vanishes. We adopt the same convention when $\alpha=1$, in which case the negative powers $(1-\alpha)^{m-1-i-k}$ would otherwise be undefined.
Now, from Lemma \ref{L3} we have
\begin{align*}
\PP\!\bigl(X_m>k\bigr)&=\alpha\mbox{$\sum_{i=0}^{m-1}\binom{m-1-i}{k}\alpha^k(1\!-\!\alpha)^{m-1-i-k}$},\\
\PP\!\bigl(X_{m-j}>k\bigr)&=\alpha\mbox{$\sum_{i=0}^{m-1-j}\binom{m-1-i-j}{k}\alpha^k(1\!-\!\alpha)^{m-1-j-i-k}$}\\
&=\alpha\mbox{$\sum_{i=j}^{m-1}\binom{m-1-i}{k}\alpha^k(1\!-\!\alpha)^{m-1-i-k}$}.
\end{align*}
Subtracting both equalities we obtain \eqref{eq:ultima}.
\end{proof}

\begin{lemma}\label{L77}
Let $X\sim \Bin(n,\alpha)$ and $Y\sim \Bin(n,r)$  independent with $\palpha=\alpha\lip/\lipalpha$ and  $n\in\N$.
Then
\[
        \PP(X\!=\!Y\!-\!1)=\lip\,\PP(X\!=\!Y\!+\!1).
\]
\end{lemma}
\begin{proof}
If $n=0$ or $\alpha=1$ both sides are 0, so we may assume $n\geq 1$ and $\alpha\in (0,1)$. For \(0\leq k\leq n-1\),
\begin{align*}
\PP(X=k,Y=k+1)&=\mbox{$\binom{n}{k}\binom{n}{k+1}\alpha^k(1-\alpha)^{n-k}r^{k+1}(1-r)^{n-k-1}$}, \\
\PP(X=k+1,Y=k)&=\mbox{$\binom{n}{k+1}\binom{n}{k}\alpha^{k+1}(1-\alpha)^{n-k-1}
r^k(1-r)^{n-k}$},
\end{align*}
and therefore, since $\palpha/(1-\palpha)=\lip\,\alpha/(1-\alpha)$, we obtain
$$ \frac{\PP(X=k,Y=k+1)}{\PP(X=k+1,Y=k)}=\frac{r(1-\alpha)}{\alpha(1-r)}= \lip.
$$
Summing over \(k=0,\ldots,n-1\) gives $\PP(X\!=\!Y\!-\!1)=\lip\,\PP(X\!=\!Y\!+\!1)$.
\end{proof}

We may now proceed to prove the announced domination.
\begin{proposition}\label{P11}
For every $0\le j\le m-1$ we have $       \PP(V^j_{m,m+1})\leq \PP(H^j_{m,m+1})$.
\end{proposition}
\begin{proof}
For $j=0$ this is trivial since $\PP(V^0_{m,m+1})=0$, so let us assume $j\geq 1$.
From \eqref{pmnij1}-\eqref{pmnij2} we can write $\PP(V_{m,m+1}^j)=\alpha\lip(A-B)$ with
\begin{align*}
     A&=\mbox{$\alpha\sum_{k\geq 0}\binom{m-j}{k}(\alpha\lip)^{k}(1-\alpha)^{m-j-k}\sum_{i=0}^{j-1}\binom{m-1-i}{k}\alpha^{k}(1-\alpha)^{m-1-i-k}$},\\
     B&=\mbox{$\alpha\sum_{k\geq 0}\binom{m-1-j}{k}(\alpha\lip)^{k}(1-\alpha)^{m-1-j-k}\sum_{i=0}^{j-1}\binom{m-i}{k}\alpha^{k}(1-\alpha)^{m-i-k}$}.
\end{align*}
Then, using Lemma \ref{L5} and proceeding as in the proof of Proposition \ref{Prop4}, we obtain
\begin{align*}
     A &=\mbox{$\sum_{k\geq 0}\binom{m-j}{k}(\alpha\lip)^{k}(1-\alpha)^{m-j-k}\big(\PP(X_m>k)-\PP(X_{m-j}>k)\big)$}\\
       &=\lipalpha^{\hspace{0.2ex}m-j}\mbox{$\big(\PP(X_m> Y_{m-j})-\PP(X_{m-j}>Y_{m-j})\big)$},\\[1ex]
     B &=\mbox{$\sum_{k\geq 0}\binom{m-1-j}{k}(\alpha\lip)^{k}(1-\alpha)^{m-1-j-k}\big(\PP(X_{m+1}> k)-\PP(X_{m+1-j}>k)\big)$},\\
       &=\lipalpha^{\hspace{0.2ex}m-1-j}\mbox{$\big(\PP(X_{m+1}> Y_{m-1-j})-\PP(X_{m+1-j}>Y_{m-1-j})\big)$}.
\end{align*}
Setting $b=m-j$, the previous formula for $\PP(V^j_{m,m+1})$ and \eqref{pdoti} for $\PP(H^j_{m,m+1})$, give
\begin{align*}
\PP(V^j_{m,m+1})&=\alpha\lip\left[\lipalpha^{\hspace{0.2ex}b} \left(\PP(X_m\!>\!Y_b)-\PP(X_b\!>\!Y_b)\right)-
 \lipalpha^{\hspace{0.2ex}b-1} \left(\PP(X_{m+1}\!>\!Y_{b-1})-\PP(X_{b+1}\!>\!Y_{b-1})\right)\right],\\
\PP(H^j_{m,m+1})&=\alpha\lip\left[\lipalpha^{\hspace{0.2ex}b-1}\PP(X_b\!>\!Y_{b-1})-\lipalpha^{\hspace{0.2ex}b}\PP(X_{b-1}\!>\!Y_b)\right].
\end{align*}

Let us show that $\Delta\triangleq\big(\PP(H^j_{m,m+1})-\PP(V^j_{m,m+1})\big)/(\alpha\lip\, \lipalpha^{\hspace{0.2ex}b-1})$
is nonnegative. Denoting for simplicity $X=X_{b-1}$, $M= X_m$, and $Y=Y_{b-1}$, we have
\begin{align*}
\Delta&=
\PP(X_b\!>\!Y)-\lipalpha\,\PP(X\!>\!Y_b)-\lipalpha\,\PP(M\!>\!Y_b)+\lipalpha\,\PP(X_b\!>\!Y_b)+\PP(X_{m+1}\!>\!Y)-\PP(X_{b+1}\!>\!Y).
\end{align*}
Using the identities $X_b=X+B_b$ and $Y_b=Y+D_b$ we can expand
\begin{align*}
\PP(X_b\!>\!Y)& =\PP(X\!>\!Y)+\alpha\,\PP(X\!=\!Y),\\
\PP(X\!>\!Y_b)&=\PP(X\!>\!Y)-\palpha\,\PP(X\!=\!Y\!+\!1),\\
\PP(M\!>\!Y_b)&=\PP(M\!>\!Y)-\palpha\,\PP(M\!=\!Y\!+\!1).
\end{align*}
Similarly,
\begin{align*}
\PP(X_b\!>\!Y_b)
&=\PP(X\!>\!Y_b)+\alpha\,\PP(X\!=\!Y_b) \\
&=\PP(X\!>\!Y)-\palpha\,\PP(X\!=\!Y\!+\!1)+\alpha(1-\palpha)\,\PP(X\!=\!Y)+\alpha\, \palpha\,\PP(X\!=\!Y\!+\!1) \\
&=\PP(X\!>\!Y)+\alpha(1-\palpha)\PP(X\!=\!Y)-\palpha(1-\alpha)\PP(X\!=\!Y\!+\!1),\\[1ex]
\PP(X_{m+1}\!>\!Y)&=\PP(M\!>\!Y)+\alpha\,\PP(M\!=\!Y).
\end{align*}
Finally, since $X_{b+1}=X+Z$ where $Z\triangleq B_b+B_{b+1}$
satisfies
$\PP(Z\!=\!0)=(1\!-\!\alpha)^2$, $\PP(Z\!=\!1)=2\alpha(1\!-\!\alpha)$ and
$\PP(Z\!=\!2)=\alpha^2$, we get
$$\PP(X_{b+1}\!>\!Y)=\PP(X\!>\!Y)+\alpha(2-\alpha)\PP(X\!=\!Y)+\alpha^2\PP(X\!=\!Y\!-\!1).
$$
Substituting all these identities into the expression of $\Delta$ and recalling that  $\lipalpha=1-\alpha+\alpha\,\lip$, and $\palpha=\alpha\,\lip/\lipalpha$ so that $\alpha\,\lipalpha\,(1-r)=\alpha(1-\alpha)$, the expression of $\Delta$ simplifies to
\begin{align*}
\Delta &=\alpha\big[\PP(M\!=\!Y) +(1-\lip)\,\PP(M\!>\!Y)  +\lip\,\PP(M\!=\!Y\!+\!1)\big]+\alpha^2\big[\lip\,\PP(X\!=\!Y\!+\!1)-\PP(X\!=\!Y\!-\!1)\big].
\end{align*}
To conclude, we note that the first bracket is nonnegative, while Lemma \ref{L77}  shows that the expression in the second bracket vanishes. 
\end{proof}

\section{Alternative expressions for \texorpdfstring{$p_{m,n}^{i,j}$}{pmnij}}\label{App:Alt_pmnij}

Consider the Markov chain in \S{}\ref{Sec:2.2}, and let  $p_{m,n}^{i,j}$ be the probability that, starting from $(m,n)\in \D$, the chain visits a state $(i,j)\in \D$ with $0\le i<m$ and $i< j<n$. As noted in \S \ref{sec:6}, this 
is given by \eqref{pmnij1}-\eqref{pmnij2} where
the 
sum \eqref{pmnij1} is finite since $L_{m,n}^{k,i,j}=0$ for $k\geq\min\{n-j,m-i\}$. 
For $\alpha\in(0,1)$, adapting the proof of Proposition \ref{Thm:Alt_cmn}, with $\zalpha=(\frac{\alpha}{1-\alpha})^2\lip$ and
  $A_{m,n}^{i,j}\triangleq(1\!-\!\alpha)^{m+n-i-j-2}$, we get
\begin{equation*}
p_{m,n}^{i,j}=
\mbox{$\alpha^2\lip A_{m,n}^{i,j}\left(\Gauss(j+1-n,i+1-m;1;\zalpha)-\mathbbm{1}_{\{j<m\}}\Gauss(i+1-n,j+1-m;1;\zalpha)\right)$}.
\end{equation*}
Alternatively, we can express $p_{m,n}^{i,j}$ in terms of binomial distributions, valid for all $\alpha\in(0,1]$.
\begin{proposition}Let $\pp=1-\alpha+\alpha\,\sigma$, $\pq=\alpha\sigma/\pp$, and let 
$M_a\sim\Bin(a,\pq)$
and $N_b\sim\Bin(b,\pq)$ denote independent binomial variables. Then for $0\le m<n$, $0\le i<m$, and $i<j<n$, we have
\begin{equation}\label{Eq:p_mn_ij}
p_{m,n}^{i,j}=
\alpha^2\lip\pp^{m+n-i-j-2}\left(\PP(M_{n-1-j}=N_{m-1-i})-\mathbbm{1}_{\{j<m\}}\,\PP(M_{n-1-i}=N_{m-1-j})\right).
\end{equation}
\end{proposition}
\begin{proof} Since $\pq^{2k}(1-\pq)^{a+b-2k}=\pp^{-(a+b)}(\alpha^2\lip)^k(1-\alpha)^{a+b-2k}$ we have
\begin{align*}
   \PP(M_a=N_b)&=\mbox{$\sum_{k\geq 0}\binom{a}{k}\binom{b}{k}\pq^{2k}(1-\pq)^{a+b-2k}$}\\
    & =\mbox{$\pp^{-(a+b)}\,\sum_{k\geq 0}\binom{a}{k}\binom{b}{k}(\alpha^2\lip)^k(1-\alpha)^{a+b-2k}$}.
\end{align*}
Using this identity twice, first with 
$a=n-1-j, b=m-1-i$, and then $a=n-1-i, b=m-1-j$,
and noting that in both cases $a+b=m+n-i-j-2$,
it follows that \eqref{Eq:p_mn_ij} coincides with \eqref{pmnij1}.
\end{proof}

From \eqref{Eq:p_mn_ij} we can derive formulas for $p_{m,n}^{i,j}$ in terms of trigonometric integrals. Namely,
denoting $f(\theta)\triangleq 1-\pq+\pq\, e^{\ii\,\theta}$, the characteristic function of $M_a-N_b$ is
$\EE\big[e^{\ii\,\theta (M_a-N_b)}\big]=f(\theta)^af(-\theta)^b$
so that using the inversion formula (see {\em e.g.} \citet[p.~511]{feller1971vol2}) we have
$$\PP(M_a=N_b)=\frac{1}{2\pi}\int_{-\pi}^\pi f(\theta)^af(-\theta)^b\,d\theta.$$
Then, setting $z_{\alpha}(\theta)\triangleq \pp\,f(\theta)=1-\alpha+\alpha\,\sigma\, e^{\ii\,\theta}=r_{\!\alpha}(\theta)e^{\ii\,\varphi_{\!\alpha}(\theta)}$ as in \S{}\ref{sec:trigo}, the identity 
\eqref{Eq:p_mn_ij} becomes
\begin{align*}
    p_{m,n}^{i,j}&=
\frac{\alpha^2\lip}{2\pi}\!\int_{-\pi}^\pi\left(z_{\alpha}(\theta)^{n-1-j}z_\alpha(-\theta)^{m-1-i}-\mathbbm{1}_{\{j<m\}}\,z_{\alpha}(\theta)^{n-1-i}z_\alpha(-\theta)^{m-1-j}\right)\,d\theta\\
    &=\frac{\alpha^2\lip}{\pi}\!\!\int_{0}^\pi r_{\!\alpha}(\theta)^{n+m-i-j-2}\big[ \cos((n\!-\!m\!+\!i\!-\!j)\varphi_{\!\alpha}(\theta))-\mathbbm{1}_{\{j<m\}}\,\cos((n\!-\!m\!+\!j\!-\!i)\varphi_{\!\alpha}(\theta))\big]\,d\theta.
\end{align*}

\section{Iteration  \texorpdfstring{\eqref{eq:km_a}}{kmconv} in convex metric spaces} \label{App:ConvexMetricSpace}

Convex metric spaces are a generalization of Busemann convex spaces, which themselves contain all CAT(0) spaces, including $\RR$-trees and Hadamard manifolds.  Specifically, a metric space $(X,d)$ is said to be a {\em convex metric space} in the sense of \citet{Takahashi1970}, when there is a map $W:X\times X\times [0,1]\to X$
such that for all points $x,y,z\in X$ and $t\in[0,1]$
\begin{equation}\label{conv}
d(W(x,y,t),z)\leq (1-t)\,d(x,z)+t\, d(y,z).
\end{equation}
Using this inequality with $z=x$ and $z=y$ we get
\begin{align}
d(W(x,y,t),x)&\leq t\, d(y,x),\label{W1}\\
d(W(x,y,t),y)&\leq (1-t)\, d(x,y),\label{W2}
\end{align}
and the triangle inequality $d(x,y)\leq d(x,W(x,y,t))+d(W(x,y,t),y)$ implies that \eqref{W1} and \eqref{W2}  must hold with equality.  Thus $W(x,y,t)$ is located at the proportional distances from $x$ and $y$, and in particular $W(x,y,0)=x$ and $W(x,y,1)=y$. 

The map $W$ plays the role of convex combinations in regular normed spaces.
For this reason, and for notational convenience, we denote $W(x,y,t)$ as $(1-t)\,x\oplus t\,y$.
A subset $C\subseteq X$ is called {\em $W$-convex} if $(1-t)\,x\oplus t\,y\in C$ for all $x,y\in C$ and $0\leq t\leq 1$.

We will show that Theorem \ref{thm:main} is valid on convex metric spaces. Let $C$ be $W$-convex with finite diameter $\mathop{\rm diam}(C)\leq\kappa$, and $T:C\to C$ an $\lip$-Lipschitz map with $\lip\in(0,1]$. Given an initial point $x_0\in C$ and a (possibly) nonconstant sequence $\alpha_m\in(0,1]$ with $\alpha_0=1$, we consider the iteration
\begin{equation}\label{kmconv}
(\forall m\in\N)\qquad x_{m+1}=(1-\alpha_{m+1})\,x_{m}\oplus\alpha_{m+1}\,Tx_{m},
\end{equation}

 The following adapts  \citet[Proposition 4.4]{foglia_colao_2025_cat0} to convex metric spaces and nonconstant $\alpha_m$'s. We denote
 $\pi_i^m=\alpha_i\prod_{k=i+1}^m(1-\alpha_k)$ for $0\leq i\le m$, and use the convention $Tx_{-1}=x_0$.
  
\begin{proposition}\label{P19}
For all $0\le m\le n$ we have $d(x_m,x_n)\leq\sum_{i=0}^m\sum_{j=m+1}^n\pi_i^m\pi_j^n\,d(Tx_{i-1},Tx_{j-1})$.
\end{proposition}
\begin{proof}
Let us first show inductively that for fixed $m$ we have 
\begin{equation}\label{part}
(\forall n\geq m)\qquad d(x_n,x_m)\leq\sum_{j=m+1}^n\pi_j^n\,d(Tx_{j-1},x_m).
\end{equation}
This is clear for $n=m$ since both sides are 0. Using  \eqref{conv} and \eqref{kmconv},  an induction step then gives 
\begin{align*}
    d(x_{n+1},x_m)&\leq (1-\alpha_{n+1})\,d(x_n,x_m)+\alpha_{n+1}\,d(Tx_n,x_m)\\
    &\leq \mbox{$ (1-\alpha_{n+1})\sum_{j=m+1}^n\pi_j^n\, d(Tx_{j-1},x_m)+\alpha_{n+1}\,d(Tx_n,x_m)$}\\
    &=\mbox{$\sum_{j=m+1}^{n+1}\pi_j^{n+1} d(Tx_{j-1},x_m)$}.
\end{align*}
Analogously, an induction in $m$ shows that for fixed $z\in X$ one has
$$d(z,x_m)\leq\sum_{i=0}^m\pi_i^m\,d(z,Tx_{i-1}).$$
Applying this with $z=Tx_{j-1}$ and plugging it  into \eqref{part} yields the conclusion.
\end{proof}

We may now derive the announced extension of Theorem \ref{thm:main} to convex metric spaces.

\begin{corollary}
Consider the iteration \eqref{kmconv}
with $\alpha_m\equiv\alpha$ for all $m\geq 1$.
Let $\cm_{m,n}$ and $\bm_m$  be defined recursively as in \emph{\S\ref{sec:recursive}}. Then, for all $0\le m\le n$ we have
$d(x_m,x_n)\leq\kappa\,\cm_{m,n}$ 
and $d(x_m,Tx_m)\leq \kappa\,\bm_m$.
\end{corollary}
\begin{proof} 
The  estimate $d(x_m,x_n)\leq\kappa\,\cm_{m,n}$  follows recursively exactly as in \S\ref{sec:recursive} by using the inequality in Proposition \ref{P19}, with the coarse bound $d(x_0,Tx_{j-1})\leq\kappa$ for the terms with $i=0$, and  the Lipschitz inequality $d(Tx_{i-1},Tx_{j-1})\leq\lip\,d(x_{i-1},x_{j-1})$ for the remaining terms.

Next, since \eqref{W1} holds with equality, by taking $x=x_m$, $y = T x_m$ and $t =\alpha$, it follows that
$d(x_m,x_{m+1})=\alpha\,d(x_m,Tx_m)$,
and therefore $d(x_m,Tx_m)\leq\kappa\,\cm_{m,m+1}/\alpha=\kappa\,\bm_m$.
\end{proof}

\section{Comparison with tight error bounds}\label{App:Compare_dmn}

In this section we compare the $c_{m,n}$'s with the tight optimal transport bounds $d_{m,n}$ in \S\ref{Sec:optimal_transport}. This extends the result for $\lip=1$ in \cite[Proposition 3.1]{bc18}. We keep $\lip\in (0,1]$ but consider $\alpha\in[\frac12,1]$ in which case the optimal transports 
are given by \eqref{trasporte_optimo}. Hence, $d_{m,m}=0$ and $d_{0,n}=1-\pi_0^n$, while for $n>m\geq 1$ we get
\begin{equation}\label{eq:dmn}
d_{m,n} = \pi_0^m\!-\pi_0^n+\sum_{i=1}^{m-1}\lip(\pi_i^m\!-\pi_i^n)d_{i-1,n-1}+\!\!\!\sum_{j=m+1}^{n-1}\!\!\!\lip\pi_j^nd_{m-1,j-1}+\mbox{$\lip(\pi_m^m-\sum_{j=m}^{n-1}\pi_j^n)$}d_{m-1,n-1}.
\end{equation}
Similarly to the $\cm_{m,n}$'s, this recursion can be interpreted as absorption probabilities for an alternative Markov chain where $\st,\ts$ are absorbing with 
    $Q(\st; \st)=Q(\ts; \ts)=1$, and the transitions from a state $(m,n)\in\D$ 
    are now $Q((m,m);\st)=1$ if $n=m$, while for $n>m$ we have
\begin{equation*}\label{Eq:probQ}
\begin{cases}
   Q((m,n); (i-1,n-1))=\lip \,(\pi_{i}^m -\pi_{i}^n)&  \text{for }~~i=1,\ldots,m-1\\
   Q((m,n); (m-1,j-1))=\lip \,\pi_{j}^n&  \text{for }~~j=m+1,\ldots,n-1\\
     Q((m,n); (m-1,n-1))=\lip \,(\pi_m^m-
   \sum_{j=m}^{n-1}\pi_j^n)& \\
      Q((m,n); \ts)=\pi_{0}^m -\pi_{0}^n&\\
   Q((m,n);\st)=\mbox{$1-(\pi_{0}^m -\pi_{0}^n)-\lip\sum_{i=1}^{m}(\pi_i^m-\pi_{i}^n)$}.&
\end{cases}
\end{equation*}
All formulas involving $m-1$ are for $m \ge 1$, while $m = 0$ is a
separate boundary case.
In fact, for $m=0$ the first 3 cases above are void so that
$Q((0,n); \st)=\pi_{0}^n$ and $Q((0,n); \ts)=1-\pi_0^n$. Note that the latter
 coincide with the transitions of the $P$-chain \eqref{Eq:probP} in Section \ref{Sec:2.2}. Similarly to the $P$-chain one can check that all the $Q$ entries are nonnegative and every row adds up to 1. Also, for $n=m$ we have $Q((m,m);\st)=1=P((m,m);\st)$.

Now, in every step, both chains  jump to a new transient state in $\D$ with strictly smaller coordinates, or are absorbed at $\st$ or $\ts$. Hence,
starting from $(m,n)$ absorption occurs after at most $m+1$ steps.
Denoting $e=(e_s)_{s\in \mathcal{S}}$  with $e_{\ts }=1$ and $e_s=0$ for $s\ne\ts$, and setting $c^{(k)}=P^ke$ and $d^{(k)}=Q^ke$,
we have that  $c^{(k)}_{m,n}$ and $d^{(k)}_{m,n}$ are the probabilities that starting from $(m,n)$ the chains are absorbed at $\ts$ by time $k\in\N$, and therefore
$\cm_{m,n}=c^{(m+1)}_{m,n}$ and $d_{m,n}=d^{(m+1)}_{m,n}$\!.
In order to estimate $\cm_{m,n}-d_{m,n}$, we first establish two technical lemmas.
For simplicity hereafter we denote $\beta\triangleq 1-\alpha$.
\begin{lemma}
\label{lem:increasing-differences1}
If $\alpha\in[\frac12,1]$ then $0\leq d^{(k)}_{h,l}-d^{(k)}_{h,j} \leq d^{(k)}_{i,l}-d^{(k)}_{i,j}$ for all $0\le h\le i\le j\le l$ and  $k\in\N$.
\end{lemma}
\begin{proof} 
Denoting $\Delta^{\!(k)}_{m,n}\triangleq d^{(k)}_{m,n+1}\!-d^{(k)}_{m,n}$, by telescoping we get $d^{(k)}_{h,l}-d^{(k)}_{h,j}=\sum_{u=j}^{l-1}\Delta^{\!(k)}_{h,u}$ so it suffices to show that 
for all $k\in\N$ the differences $\Delta^{\!(k)}_{m,n}$ are nonnegative and 
increase with $m$ for $m\le n$. 

For $k=0$ this holds trivially since $d^{(0)}=e$ 
so that $\Delta^{\!(0)}_{m,n}=e_{m,n+1}-e_{m,n}\equiv 0$. 
Inductively, suppose that the property holds for a given $k$ and let us prove it for $k+1$. 
Since $\st$ and $\ts$ are absorbing, we have  $d^{(k)}_{\st}=d^{(k-1)}_{\st}=\ldots=d^{(0)}_{\st}=e_{\st}=0$
    and $d^{(k)}_{\ts}=d^{(k-1)}_{\ts}=\ldots=d^{(0)}_{\ts}=e_{\ts}=1$,  hence
    $$d^{(k+1)}_{0,n}=(Qd^{(k)})_{0,n}=\pi_0^nd^{(k)}_{\st}+(1-\pi_{0}^n)d^{(k)}_{\ts}=1-\pi_{0}^n,$$ 
and therefore
$\Delta^{\!(k+1)}_{0,n}=\pi_0^n-\pi_0^{n+1}\geq 0$. 
It remains to show that $\Delta^{\!(k+1)}_{m,n}\geq \Delta^{\!(k+1)}_{m-1,n}$ for $1\leq m\le n$. 
    
From $d^{(k+1)}=Qd^{(k)}$ we have
\begin{align}\label{ah}
d^{(k+1)}_{m,n}&=\mbox{$(\pi_{0}^m -\pi_{0}^n)+\lip\sum_{j=m+1}^{n-1}\pi_{j}^n\big[d^{(k)}_{m-1,j-1}\!-d^{(k)}_{m-1,n-1}\big]+ \lip\sum_{i=1}^{m}(\pi_{i}^m -\pi_{i}^n)\,d^{(k)}_{i-1,n-1}$},
\end{align}
from which it follows that
\begin{align*}
\Delta^{(k+1)}_{m,n}=&~\mbox{$(\pi_{0}^n -\pi_{0}^{n+1})+\lip\sum_{j=m+1}^{n}(\pi_j^n-\pi_{j}^{n+1})\big[d^{(k)}_{m-1,n-1}\!-d^{(k)}_{m-1,j-1}\big]+ \lip\sum_{i=1}^{m}(\pi_i^n-\pi_{i}^{n+1})\,d^{(k)}_{i-1,n}$}\\
&+\mbox{$\lip\sum_{j=m+1}^{n}\pi_{j}^{n+1}\big[d^{(k)}_{m-1,n-1}\!-d^{(k)}_{m-1,n}\big]+ \lip\sum_{i=1}^{m}(\pi_{i}^m -\pi_{i}^n)\big[d^{(k)}_{i-1,n}-d^{(k)}_{i-1,n-1}\big].$}
\end{align*}
In particular, for $m=1$ we get
\begin{align*}
\Delta^{\!(k+1)}_{1,n}-\Delta^{\!(k+1)}_{0,n}=&~\mbox{$\lip\sum_{j=2}^{n}(\pi_{j}^n-\pi_j^{n+1})\big[d^{(k)}_{0,n-1}\!-d^{(k)}_{0,j-1}\big]+\lip\,(\pi_{1}^n -\pi_{1}^{n+1})\,d^{(k)}_{0,n}$}\\
&+\mbox{$\lip\,(\pi_{1}^1 -\sum_{j=1}^n\pi_j^{n+1})\big[d^{(k)}_{0,n}-d^{(k)}_{0,n-1}\big]$},
\end{align*}
with all the terms nonnegative
since the induction hypothesis gives $d^{(k)}_{0,j-1}\le d^{(k)}_{0,n-1}\le d^{(k)}_{0,n}$, and we also have $(\pi_j^n-\pi_{j}^{n+1})\ge 0$ and $(\pi_1^1-\sum_{j=1}^n\pi_j^{n+1})= 2\alpha-1+\beta^{n+1}\geq 0$. Therefore  
$\Delta^{\!(k+1)}_{1,n}\ge \Delta^{\!(k+1)}_{0,n}$.

Consider now the case $2\le m \le n$. 
A routine calculation yields 
\begin{align*}
\Delta^{(k+1)}_{m,n}-\Delta^{(k+1)}_{m-1,n}=
&~\mbox{$\lip\sum_{j=m+1}^{n-1}(\pi_j^n-\pi_{j}^{n+1})\big[(d^{(k)}_{m-1,n-1}\!-d^{(k)}_{m-1,j-1})-(d^{(k)}_{m-2,n-1}\!-d^{(k)}_{m-2,j-1})\big]$}\\
&+\mbox{$\lip\sum_{i=1}^{m-1}(\pi_{i}^{m-1} -\pi_{i}^{m})\big[(d^{(k)}_{m-2,n}\!-d^{(k)}_{m-2,n-1})-(d^{(k)}_{i-1,n}\!-d^{(k)}_{i-1,n-1})\big]$}\\
&+\mbox{$\lip\big(\pi_m^m-\pi_{m}^n-\sum_{j=m+1}^{n}\pi_{j}^{n+1}\big)\big[(d^{(k)}_{m-1,n}\!-d^{(k)}_{m-1,n-1})-(d^{(k)}_{m-2,n}\!-d^{(k)}_{m-2,n-1})\big]$}\\
&+\mbox{$\lip(\pi_0^{m-1}-\pi_0^m)\big[d^{(k)}_{m-2,n}\!-d^{(k)}_{m-2,n-1}\big]$}\\
&+\mbox{$\lip(\pi_m^n-\pi_{m}^{n+1})\big[(d^{(k)}_{m-1,n}\!-d^{(k)}_{m-1,m-1})-(d^{(k)}_{m-2,n}\!-d^{(k)}_{m-2,m-1})\big]$}
\end{align*}
where in the last line we added the term $d^{(k)}_{m-1,m-1}\equiv 0$, an identity that follows since diagonal states are absorbed at $\st$.
From the induction hypothesis all the expressions in square brackets are nonnegative, as well as the corresponding coefficients since $(\pi_j^n-\pi_{j}^{n+1})\geq 0$ and $(\pi_{i}^{m-1}- \pi_{i}^{m})\geq 0$, while
$\alpha\geq\frac12$ implies 
\begin{align*}
   \mbox{$\big(\pi_{m}^{m}-\pi_{m}^n-\sum_{j=m+1}^{n}\pi_{j}^{n+1}\big)$}&=(2\alpha-1)(1-\beta^{n-m})\geq 0.
\end{align*}
Hence $\Delta^{\!(k+1)}_{m,n}\geq \Delta^{\!(k+1)}_{m-1,n}$ which completes the proof.
\end{proof}

\begin{lemma}
\label{lem:local}
Suppose that $\alpha\in[\frac12,1]$ and define $h^{(k)}\triangleq (P-Q)d^{(k)}$.
Then, for all $k\in\N$ and $(m,n)\in\D$ we have $d_{m,n}^{(k)}\leq \lipalpha^{\hspace{0.3ex}m}$
and $0\le h^{(k)}_{m,n}\leq (1-\alpha)^2\lipalpha^{\hspace{0.3ex}m}$.
\end{lemma}
\begin{proof}
We prove  inductively that $d_{m,n}^{(k)}\leq \lipalpha^{\hspace{0.3ex}m}$.
    This holds trivially for $k=0$ since $d_{m,n}^{(0)}=e_{m,n}=0$. Assume $d_{i,j}^{(k)}\leq \lipalpha^i$ for all $(i,j)\in\D$ and let us estimate $d_{m,n}^{(k+1)}$. According to Lemma \ref{lem:increasing-differences1}, we have $d_{m-1,j-1}^{(k)}\leq d_{m-1,n-1}^{(k)}$ for all $j=m+1,\ldots,n$, and then  neglecting the first sum in \eqref{ah} we get
    \begin{align*}
d^{(k+1)}_{m,n}&\le\mbox{$\lip\sum_{i=1}^{m}(\pi_{i}^m -\pi_{i}^n)\,d^{(k)}_{i-1,n-1}+(\pi_{0}^m -\pi_{0}^n)$}.
\end{align*}
    Dropping the negative terms and using the induction hypothesis, a direct calculation gives
     \begin{align*}
        d_{m,n}^{(k+1)}&\le\mbox{$\lip\sum_{i=1}^{m}\pi_{i}^m \lipalpha^{i-1}+\pi_{0}^m=\lipalpha^m$},
    \end{align*} 
    which completes the induction step.

Let us now show that $0\le h^{(k)}_{m,n}\leq (1-\alpha)^2\lipalpha^{\hspace{0.3ex}m}$.

For $m=0$ and $m=n$ the $(m,n)$-th rows of $P$ and $Q$
coincide, and therefore $h^{(k)}_{m,n}=0\le\beta^2\lipalpha^{\hspace{0.2ex}m}$. 
Let us consider the case
 $0< m< n$. The transition probabilities towards $\ts$ and $\st$ are equal in $P$ and $Q$, so only transitions to transient states contribute to $h^{(k)}_{m,n}$. After simplification we get
\begin{equation}
\label{eq:local-decomposition}
\begin{aligned}
  h^{(k)}_{m,n}
  =&\;\sum_{i=0}^{m-2}
      \sum_{j=m}^{n-2}\lip\pi_{i+1}^m\pi_{j+1}^n
      \big(d^{(k)}_{i,j}-d^{(k)}_{i,n-1}\big) +\lip\beta\sum_{j=m}^{n-2}\pi_{j+1}^n
      \big(d^{(k)}_{m-1,n-1}-d^{(k)}_{m-1,j}\big).
\end{aligned}
\end{equation}
Using again Lemma \ref{lem:increasing-differences1}, for $i=0,\ldots,m-2$ and
$j=m,\ldots, n-2$ we have
\[
  0\leq d^{(k)}_{i,n-1}-d^{(k)}_{i,j}
  \leq d^{(k)}_{m-1,n-1}-d^{(k)}_{m-1,j},
\]
while
$\sum_{i=0}^{m-2}\pi_{i+1}^m
  =\beta-\beta^{m}
  \leq\beta$. Therefore, the first double sum in \eqref{eq:local-decomposition} is negative and bounded in absolute value by the second sum, so that
\begin{align*}
 0\leq  h^{(k)}_{m,n}
  &\mbox{$\leq\lip\,\beta\sum_{j=m}^{n-2}\pi_{j+1}^n
      \bigl(d^{(k)}_{m-1,n-1}-d^{(k)}_{m-1,j}\bigr)$}.
\end{align*}
Dropping the negative term $-d^{(k)}_{m-1,j}$ and using the bound $d^{(k)}_{m-1,n-1}\leq \lipalpha^{\hspace{0.2ex}m-1}$,
 we get
\begin{align*}
0\leq h^{(k)}_{m,n}&\leq \mbox{$\lip\,\beta\,\lipalpha^{\hspace{0.2ex}m-1}\sum_{j=m}^{n-2}\pi_{j+1}^{n}$},
\end{align*}
and the conclusion follows by noting that $\lip\leq\lipalpha$ and 
$\sum_{j=m}^{n-2}\pi_{j+1}^n
  =\beta-\beta^{n-m}
  \leq\beta$.
\end{proof}

We may now establish the announced comparison between the $\cm_{m,n}$'s and the tight bounds $d_{m,n}$.
\begin{proposition}\label{cmndmn} For all $\lip\in (0,1]$, $\alpha\in[\frac12,1]$, and $0\le m<n$, we have $$0\le c_{m,n}-d_{m,n}\leq (1-\alpha)^2(m+1)\lipalpha^{\hspace{0.3ex}m}.$$
For $n = m+1$, equation \eqref{eq:dmn} factors as $\alpha[\pi_0^m + \sum_{i=1}^m \pi_i^m \lip d_{i-1,m}] =\alpha\, \Rtight_m$ and therefore
$$\bm_m-\Rtight_m\le (1-\alpha)^2(m+1)\lipalpha^{\hspace{0.3ex}m}/\alpha.$$
\end{proposition}

\begin{proof}
Set $\beta=1-\alpha$ as before. From Lemma~\ref{lem:local} we have the componentwise inequality $0\le h^{(k)}\leq \beta^2 w$ where 
$w=(w_s)_{s\in S}$ with $w_{m,n}=\lipalpha^{\hspace{0.3ex}m}$ for $(m,n)\in\D$ and $w_{\st}=w_{\ts}=0$. 
We claim that $w$ is $P$-superharmonic, that is $Pw\leq w$. Indeed, since $w_{\st}=w_{\ts}=0$, only transient states contribute to $Pw$ so that $(Pw)_{\st}=0=w_{\st}$, $(Pw)_{\ts}=0=w_{\ts}$, while for $(m,n)\in\D$ we get
\begin{align*}
  (Pw)_{m,n}
  &=\mbox{$\lip\sum_{i=0}^{m-1}\sum_{j=m}^{n-1}
      \pi_{i+1}^m\pi_{j+1}^n\lipalpha^i$}\leq \mbox{$\lip\sum_{i=0}^{m-1}\pi_{i+1}^m\lipalpha^i$} =\lipalpha^{\hspace{0.3ex}m}-\beta^m\leq w_{m,n}.
\end{align*}
Now, by telescoping we have $ P^{m+1}-Q^{m+1}
  =\sum_{k=0}^{m}P^k(P-Q)Q^{m-k}$, which applied to $e$ gives 
$$\mbox{$ c^{(m+1)}-d^{(m+1)}
  =\sum_{k=0}^{m}P^k h^{(m-k)}\le \beta^2\sum_{k=0}^{m}P^k w$}\leq \beta^2(m+1)w$$
and the conclusion follows by considering the $(m,n)$-th coordinate.
\end{proof}

\subsection{Numerical comparison}\label{AppE_Numerical}
The absolute-gap estimate in Proposition \ref{cmndmn} is rather conservative, and does not allow to bound  the relative gap $\bm_m/\Rtight_m$. In Table \ref{relgap} we compute a representative sample of absolute and relative gaps with $m=10$ and $m=100$, for different combinations of $\lip$ and $\alpha$. The values of $\bm_m$ were computed using the formulas in Theorem \ref{Prop6}, while $\Rtight_m=d_{m,m+1}/\alpha$ was obtained from the recursion \eqref{eq:dmn}. We also considered $m=1000$ and $\alpha,\lip\in\{0.5, 0.75, 0.9\}$, observing absolute gaps of the order $10^{-25}$ or smaller, and relative gaps below 3\%. 
As shown in Table \ref{relgapopt}, these relative gaps become even smaller when $\alpha$ is fixed at the value $\alpha_\lip(m)$ that minimizes $\bm_m$.

\begin{table}[ht]
\begin{tabular}{|ccc|lllc|}
\hline
$m$ & $\lip$ & $\alpha$ &\multicolumn{1}{c}{$\bm_m$} & \multicolumn{1}{c}{$\Rtight_m$} & \multicolumn{1}{c}{$\bm_m-\Rtight_m$} & $\bm_m/\Rtight_m$ \\
\hline
10   & 0.50 & 0.50 & $3.23\times10^{-2}$   & $3.15\times10^{-2}$   & $7.96\times10^{-4}$   & 1.025 \\
10   & 0.50 & 0.75 & $5.30\times10^{-3}$   & $5.23\times10^{-3}$   & $7.13\times10^{-5}$   & 1.014 \\
10   & 0.50 & 0.90 & $1.65\times10^{-3}$   & $1.64\times10^{-3}$   & $5.52\times10^{-6}$   & 1.003 \\\hline
10   & 0.75 & 0.50 & $1.14\times10^{-1}$   & $1.08\times10^{-1}$   & $6.04\times10^{-3}$   & 1.058 \\
10   & 0.75 & 0.75 & $5.83\times10^{-2}$   & $5.68\times10^{-2}$   & $1.46\times10^{-3}$   & 1.026 \\
10   & 0.75 & 0.90 & $4.53\times10^{-2}$   & $4.51\times10^{-2}$   & $1.89\times10^{-4}$   & 1.004 \\\hline
10   & 0.90 & 0.50 & $2.23\times10^{-1}$   & $2.07\times10^{-1}$   & $1.59\times10^{-2}$   & 1.077 \\
10   & 0.90 & 0.75 & $1.90\times10^{-1}$   & $1.84\times10^{-1}$   & $5.96\times10^{-3}$   & 1.032 \\
10   & 0.90 & 0.90 & $2.14\times10^{-1}$   & $2.13\times10^{-1}$   & $9.37\times10^{-4}$   & 1.004 \\
\hline
100  & 0.50 & 0.50 & $1.61\times10^{-13}$  & $1.60\times10^{-13}$  & $1.36\times10^{-16}$  & 1.001 \\
100  & 0.50 & 0.75 & $1.94\times10^{-21}$  & $1.94\times10^{-21}$  & $1.71\times10^{-24}$  & 1.001 \\
100  & 0.50 & 0.90 & $5.49\times10^{-27}$  & $5.48\times10^{-27}$  & $1.09\times10^{-29}$  & 1.002 \\\hline
100  & 0.75 & 0.50 & $4.28\times10^{-7}$   & $4.14\times10^{-7}$   & $1.38\times10^{-8}$   & 1.038 \\
100  & 0.75 & 0.75 & $2.65\times10^{-10}$  & $2.59\times10^{-10}$  & $6.11\times10^{-12}$  & 1.027 \\
100  & 0.75 & 0.90 & $2.63\times10^{-12}$  & $2.59\times10^{-12}$  & $3.88\times10^{-14}$  & 1.016 \\\hline
100  & 0.90 & 0.50 & $9.70\times10^{-4}$   & $8.73\times10^{-4}$   & $9.70\times10^{-5}$   & 1.111 \\
100  & 0.90 & 0.75 & $7.31\times10^{-5}$   & $6.87\times10^{-5}$   & $4.40\times10^{-6}$   & 1.066 \\
100  & 0.90 & 0.90 & $1.82\times10^{-5}$   & $1.78\times10^{-5}$   & $4.96\times10^{-7}$   & 1.028 \\
\hline
\end{tabular}
\caption{\label{relgap} Absolute and relative gaps for different combinations of $(m,\lip,\alpha)$}

\end{table}

Further computations with $\alpha=\frac12$ and $\lip=1$ showed relative gaps consistent with the bound $\bm_m/\Rtight_m\leq\sqrt{3/2}$, that we conjecture to be globally valid for all $m\geq1$ and  $\lip,\alpha\in[\frac12,1]$. This is also consistent with a conjecture in \cite[Section 5]{bc18} for the nonexpansive case $\lip=1$, regarding the asymptotics of $\Rtight_m$ when $\alpha=\frac12$.

\begin{table}[ht]
\begin{tabular}{|ccc|}
\multicolumn{3}{c}{$\lip=0.60$}\\[0.5ex]
\hline
$m$ & $\alpha_{\lip}(m)$ & $\bm_m/\Rtight_m$ \\
\hline
10   & 0.975958 & 1.000047 \\
100  & 0.997536 & 1.000007 \\
500  & 0.999506 & 1.000001 \\
1000 & 0.999753 & 1.000001 \\
\hline
\end{tabular}
\hspace{1cm}
\begin{tabular}{|ccc|}
\multicolumn{3}{c}{$\lip=0.75$}\\[0.5ex]
\hline
$m$ & $\alpha_{\lip}(m)$ & $\bm_m/\Rtight_m$ \\
\hline
10   & 0.915118 & 1.002703 \\
100  & 0.990709 & 1.000425 \\
500  & 0.998128 & 1.000088 \\
1000 & 0.999063 & 1.000044 \\
\hline
\end{tabular}

\vspace{1.5ex}
\begin{tabular}{|ccc|}
\multicolumn{3}{c}{$\lip=0.90$}\\[0.5ex]
\hline
$m$ & $\alpha_{\lip}(m)$ & $\bm_m/\Rtight_m$ \\
\hline
10   & 0.755654 & 1.031239 \\
100  & 0.965262 & 1.007519 \\
500  & 0.992865 & 1.001612 \\
1000 & 0.996421 & 1.000813\\
\hline
\end{tabular}
\hspace{1cm}
\begin{tabular}{|ccc|}
\multicolumn{3}{c}{$\lip=0.99$}\\[0.5ex]
\hline
$m$  & $\alpha_{\lip}(m)$ & $\bm_m/\Rtight_m$ \\
\hline
10   & 0.527960 & 1.085711 \\
100  & 0.715765 & 1.113054 \\
500  & 0.915340 & 1.036909 \\
1000 & 0.955622 & 1.019626 \\
\hline
\end{tabular}
\caption{\label{relgapopt} Relative gaps at the optimal parameters $\alpha_\lip(m)$}

\end{table}

 We complement these observations by referring to Table \ref{alphaopt}, which presents a sample of approximate values of $\alpha_\lip(m)$. Every row in that table confirms that when $\lip$ is fixed,
$\alpha_\lip(m)$ increases to 1 as $m$ grows; while each column shows that for 
$m$ fixed $\alpha_\lip(m)$ decreases towards $\frac12$ as $\lip$ grows.

\begin{table}[ht]
\begin{tabular}{|c|cccc|}
\hline
$\lip$&$m = 1$&$m = 5$&$m = 20$&$m = 100$\\\hline
0.50 &1.000000 &1.000000 &1.000000 &1.000000\\
0.51 &0.980392 &0.995939 &0.998981 &0.999796\\
0.70 &0.714286 &0.889938 &0.969122 &0.993660\\
0.90 &0.555556 &0.661519 &0.850404 &0.965262\\
0.99 &0.505051 &0.515220 &0.553262 &0.715765\\\hline
\end{tabular}

\caption{\label{alphaopt} Numerical approximation of the optimal parameters $\alpha_\lip(m)$}
\end{table}

\end{document}